\documentclass[12pt,reqno,twoside]{amsart}
\usepackage{graphicx}
\usepackage{amssymb}
\usepackage{epstopdf}
\usepackage[asymmetric,top=3.5cm,bottom=4.3cm,left=3.1cm,right=3.1cm]{geometry}
\usepackage{booktabs} 
\usepackage{array} 
\usepackage{paralist} 
\usepackage{verbatim} 
\usepackage{subfig} 
\usepackage{tabularx}
\usepackage{amsmath,amsfonts,amsthm,mathrsfs,amssymb,cite}
\usepackage[usenames]{color}
\usepackage{bm}

\newtheorem{thm}{Theorem}[section]

\theoremstyle{definition}
\newtheorem{defn}{Definition}[section]
\theoremstyle{remark}

\newtheorem{rem}{Remark}[section]
\numberwithin{equation}{section}

\newcommand{\cl}{\mathcal{L}}
\newcommand{\bu}{\mathbf{u}}
\newcommand{\bv}{\mathbf{v}}
\newcommand{\bV}{\mathbf{V}}
\newcommand{\bw}{\mathbf{w}}
\newcommand{\bff}{\mathbf{f}}

\title[3D Plasmon Resonance in Elastostatics]{On Three-dimensional Plasmon Resonance in Elastostatics}

\author{Hongjie Li}
\address{Department of Mathematics, Hong Kong Baptist University, Kowloon Tong, Hong Kong SAR.\vspace*{-4mm}}
\address{\vspace*{-4mm}and}
\address{HKBU Institute of Research and Continuing Education, Virtual University Park, Shenzhen, P. R. China.}
\email{hongjie$_{-}$li@yeah.net}

\author{Hongyu Liu}
\address{Department of Mathematics, Hong Kong Baptist University, Kowloon Tong, Hong Kong SAR.\vspace*{-4mm}}
\address{\vspace*{-4mm}and}
\address{HKBU Institute of Research and Continuing Education, Virtual University Park, Shenzhen, P. R. China.}
\email{hongyu.liuip@gmail.com; hongyuliu@hkbu.edu.hk}

\begin{document}
\maketitle

\begin{abstract}

We consider the plasmon resonance for the elastostatic system in $\mathbb{R}^3$ associated with a very broad class of sources. The plasmonic device takes a general core-shell-matrix form with the metamaterial located in the shell. It is shown that the plasmonic device in the literature which induces resonance in $\mathbb{R}^2$ does not induce resonance in $\mathbb{R}^3$. We then construct two novel plasmonic devices with suitable plasmon constants, varying according to the source term or the loss parameter, which can induce resonances. If there is no core, we show that resonance always occurs. If there is a core of an arbitrary shape, we show that the resonance strongly depends on the location of the source. In fact, there exists a critical radius such that resonance occurs for sources lying within the critical radius, whereas resonance does not occur for source lying outside the critical radius. Our argument is based on the variational technique by making use of the primal and dual variational principles for the elastostatic system, along with the highly technical construction of the associated perfect plasmon elastic waves.

\medskip

\medskip



\end{abstract}

\section{Introduction}

\subsection{Mathematical setup}

In this paper, we consider the plasmon resonance for the elastostatic system, generalising and extending our two-dimensional study in \cite{LiLiu2d} to the much more challenging three-dimensional case. Following \cite{LiLiu2d}, we first introduce the mathematical setup of the elastostatic system and plasmon resonance. Let $\mathbf{C}(x):=(\mathrm{C}_{ijkl}(x))_{i,j,k,l=1}^N$, $x\in\mathbb{R}^N$ with $N=2,3,$ be a four-rank tensor such that
 \begin{equation}\label{eq:tensor}
 \mathrm{C}_{ijkl}(x):=\lambda(x)\delta_{ij}\delta_{kl}+\mu(x)(\delta_{ik}\delta_{jl}+\delta_{il}\delta_{jk}),\ \ x\in\mathbb{R}^N,
 \end{equation}
 where $\lambda, \mu\in\mathbb{C}$ are real-valued functions, and $\delta$ is the Kronecker delta. $\mathbf{C}(x)$ describes an isotropic elastic material tensor distributed in the space $\mathbb{R}^N$, where $\lambda$ and $\mu$ are referred to as the Lam\'e constants. For a regular elastic material, it is required that the Lam\'e constants satisfy the following strong convexity condition,
 \begin{equation}\label{eq:convex}
 \mu>0\qquad\mbox{and}\qquad N\lambda+2\mu>0.
 \end{equation}
In the sequel, we write $\mathbf{C}_{\lambda,\mu}$ to specify the dependence of the elastic tensor on the Lam\'e parameters $\lambda$ and $\mu$.

Let $\Sigma$ and $\Omega$ be bounded domains in $\mathbb{R}^N$ with connected Lipschitz boundaries such that $\Sigma\Subset\Omega$. Consider an elastic parameter distribution $\mathbf{C}_{\widetilde{\lambda}, \widetilde{\mu}}$ given with
\begin{equation}\label{eq:tensor1}
\big(\widetilde{\lambda}(x), \widetilde{\mu}(x)\big)=\big(A(x)+\mathrm{i}\delta\big)(\lambda, \mu), \quad x\in\mathbb{R}^N,
\end{equation}
where $\delta\in\mathbb{R}_+$ denotes a loss parameter; $(\lambda, \mu)$ are two Lam\'e constants satisfying the strong convexity condition \eqref{eq:convex}; and $A(x)$ has a matrix-shell-core representation in the following form
\begin{equation}\label{eq:tensor2}
A(x)=\begin{cases}
+1,\qquad & x\in\Sigma,\medskip\\
c,\qquad & x\in\Omega\backslash\overline{\Sigma},\medskip\\
+1,\qquad & x\in\mathbb{R}^N\backslash\overline{\Omega},
\end{cases}
\end{equation}
where $c$ is constant that will be specified later. The choice of $c$ is critical in our study and in principle, it will be negative-valued. In such a case, $c$ is called a plasmon constant and \eqref{eq:tensor1} is referred to as a plasmonic device.

Let $\mathbf{f}$ be an $\mathbb{R}^N$-valued function that is compactly supported outside $\Omega$ satisfying
\begin{equation}\label{eq:source1}
\int_{\mathbb{R}^N}\mathbf{f}(x)\ dV(x)=0.
\end{equation}
$\mathbf{f}$ signifies an elastic source/forcing term.

Let $\mathbf{u}_\delta(x)\in\mathbb{C}^N$, $x\in\mathbb{R}^N$, denote the displacement field in the space that is occupied by the elastic configuration $(\mathbf{C}_{\widetilde{\lambda},\widetilde{\mu}},\mathbf{f})$ as described above. In the quasi-static regime, $\mathbf{u}_\delta(x)\in H_{\mathrm{loc}}^1(\mathbb{R}^N)^N$ verifies the following Lam\'e system
\begin{equation}\label{eq:lame1}
\begin{cases}
&\mathcal{L}_{\widetilde\lambda,\widetilde\mu} \mathbf{u}_\delta(x)=\mathbf{f}(x),\quad x\in\mathbb{R}^N,\medskip\\
& \mathbf{u}_\delta|_-=\mathbf{u}_\delta|_+,\quad \partial_{\bm{\nu}_{\widetilde\lambda,\widetilde\mu}} \mathbf{u}_{\delta}|_-=\partial_{\bm{\nu}_{\widetilde\lambda,\widetilde\mu}} \mathbf{u}_{\delta}|_+\quad\mbox{on}\ \ \partial\Sigma\cup\partial\Omega,\medskip\\
&\mathbf{u}_\delta(x)=\mathcal{O}\big(\|x\|^{-1}\big)\quad\mbox{as}\ \ \|x\|\rightarrow+\infty,
\end{cases}
\end{equation}
where the partial differential operator (PDO) $\mathcal{L}_{\widetilde{\lambda},\widetilde{\mu}}$ is given as follows,
\begin{equation}\label{eq:lame2}
\mathcal{L}_{\widetilde{\lambda},\widetilde{\mu}} \mathbf{u}_\delta:=\nabla\cdot\mathbf{C}_{\widetilde{\lambda},\widetilde{\mu}} {\nabla}^s\mathbf{u}_\delta=\widetilde{\mu}\Delta\mathbf{u}_\delta+(\widetilde{\lambda}+\widetilde{\mu})\nabla\nabla\cdot\mathbf{u}_\delta,
\end{equation}
with ${\nabla}^s$ defined to be the symmetric gradient
\[
{\nabla}^s\mathbf{u}_\delta:=\frac 1 2(\nabla\mathbf{u}_\delta+\nabla\mathbf{u}_\delta^T),
\]
and $T$ signifying the matrix transpose. In \eqref{eq:lame1}, the conormal derivative (or traction) is defined by
\begin{equation}\label{eq:normald}
\partial_{\bm{\nu}_{\widetilde\lambda,\widetilde\mu}}\mathbf{u}_\delta=\frac{\partial \bu_\delta}{\partial \bm{\nu}_{\widetilde\lambda,\widetilde\mu}}:= \widetilde\lambda(\nabla \cdot \ \bu_\delta)\bm{\nu}+ \widetilde\mu(\nabla \bu_\delta + \nabla \bu_\delta^T)\bm{\nu}\quad \mbox{on}\ \ \partial\Sigma\ \mbox{or}\ \partial \Omega,
\end{equation}
where $\bm{\nu}$ denotes the exterior unit normal to $\partial \Sigma/\partial \Omega$,
and the $\pm$ signify the traces taken from outside and inside of the domain $\Sigma/\Omega$, respectively.

Next, for $\mathbf{u}\in H^1_{\text{loc}}(\mathbb{R}^N)^N$ and $\mathbf{v}\in H^1_{\text{loc}}(\mathbb{R}^N)^N$, we introduce
\begin{equation}\label{eq:energy1}
\mathbf{P}_{\lambda,\mu}(\mathbf{u},\mathbf{v}):=\int_{\mathbb{R}^N}\big[\lambda(\nabla\cdot\mathbf{u})\overline{(\nabla\cdot\mathbf{v})}(x)+2\mu\nabla^s\mathbf{u}:\overline{\nabla^s\mathbf{v}}(x) \big]\ d V(x),
\end{equation}
where and also in what follows, for two matrices $\mathbf{A}=(a_{ij})_{i,j=1}^N$ and $\mathbf{B}=(b_{ij})_{i,j=1}^N$,
\begin{equation}\label{eq:matrixnorm}
\mathbf{A}:\mathbf{B}=\sum_{i,j=1}^N a_{ij}b_{ij}.
\end{equation}
For the solution $\mathbf{u}_\delta$ to \eqref{eq:lame1}, we define
\begin{equation}\label{eq:energy2}
\mathbf{E}_\delta(\mathbf{C}_{\widetilde{\lambda},\widetilde{\mu}}, \mathbf{f}):=\frac{\delta}{2}\mathbf{P}_{\lambda,\mu}(\mathbf{u}_\delta, \mathbf{u}_\delta),
\end{equation}
which is the imaginary part of
\[
\frac 1 2 \int_{\mathbb{R}^N}\overline{\nabla^s\mathbf{u}_\delta}:\mathbf{C}_{\widetilde{\lambda},\widetilde{\mu}}\nabla^s\mathbf{u}_\delta\, dV.
\]
$\mathbf{E}_\delta$ signifies the energy dissipation of the elastostatic system \eqref{eq:lame1}.

\begin{defn}\label{def:resonant}
The configuration $(\mathbf{C}_{\widetilde{\lambda},\widetilde{\mu}}, \mathbf{f})$ in \eqref{eq:lame1} is said to be \emph{resonant} if
\begin{equation}\label{eq:def1}
\lim_{\delta\rightarrow+0}\mathbf{E}_\delta(\mathbf{C}_{\widetilde{\lambda},\widetilde{\mu}},\mathbf{f})=+\infty;
\end{equation}
and it is said to be \emph{weakly resonant} if
\begin{equation}\label{eq:def2}
\limsup_{\delta\rightarrow+0}\mathbf{E}_\delta(\mathbf{C}_{\widetilde{\lambda},\widetilde{\mu}},\mathbf{f})=+\infty.
\end{equation}
\end{defn}

\subsection{Main results and connection to existing studies}

As mentioned earlier that in the elastic tensor $\mathbf{C}_{\widetilde\lambda,\widetilde\mu}$, the plasmon constant $c$ will be negatively valued. Hence, in the limiting case as the loss parameter $\delta$ goes to zero, the corresponding PDO $\mathcal{L}_{\widetilde\lambda,\widetilde\mu}$ loses its ellipticity. In this limiting case, the non-elliptic PDO $\mathcal{L}_{\widetilde\lambda,\widetilde\mu}$ possesses an infinite dimensional kernel space. Hence, ``anomalous" resonance will be induced by such an infinite dimensional kernel, which is generally referred to as plasmon resonance in the literature. It is not surprising that the resonant field demonstrates a highly oscillatory behaviour, reflected by the blowup of the associated energy; see Definition~\ref{def:resonant}. It is rather surprising that such a blowup behaviour is localized within a specific region with sharp boundaries not defined by any discontinuities in the material parameters, but the field converges to a smooth field outside that region as $\delta$ goes to zero. Another surprisingly interesting feature of the plasmon resonance is that it strongly depends on the location of the forcing source.

Plasmon materials, a.k.a. negative materials, have received significant attentions in the literature in recent years, especially in the optics. The associated plasmon resonances can find many striking applications in science and technology such as invisibility cloaking and imaging resolution enhancement. We refer to \cite{Acm13,Ack13,Ack14,AK,Bos10,Brl07,CKKL,Klsap,LLL,GWM1,GWM2,GWM3,GWM4,GWM6,GWM7,GWM8,GWM9} for the relevant study in electrostatics, \cite{ADM,AMRZ,AKL,KLO} for acoustic waves and \cite{ARYZ} for electromagnetic waves. In recent two papers \cite{AKKY,LiLiu2d}, the plasmon resonance was investigated for the elastostatic system \eqref{eq:tensor1}--\eqref{eq:lame1} in $\mathbb{R}^2$. Briefly summarising, it has been shown in \cite{AKKY,LiLiu2d} that in $\mathbb{R}^2$, when $c$ in \eqref{eq:tensor2} is a suitable negative constant being fixed, then for a broad class of forcing terms, resonance occurs. Both in \cite{AKKY} and \cite{LiLiu2d}, the dependence of the plasmon resonance on the location of the source has been shown. In \cite{AKKY}, the localized and cloaking effects of the plasmon resonance have been derived.  In this paper, we shall generalise and extend the relevant two-dimensional study to the much more challenging three-dimensional case. In summary, the following results have been achieved in the present paper.

\begin{enumerate}

\item The plasmon resonance critically associated to the nontrivial kernel of the non-elliptic PDO $\mathcal{L}_{\widetilde\lambda,\widetilde\mu}$ in the limiting case. Hence, the proper choice of the plasmon constant $c$ in \eqref{eq:tensor2} such that $\mathcal{L}_{\widetilde\lambda,\widetilde\mu}$ with $\delta=0$ possesses an infinite dimensional kernel is critical for the occurrence of the plasmon resonance. In \cite{AKKY, LiLiu2d}, the plasmon constant is a fixed negative constant, and it is derived based on the investigation of the spectral properties of the corresponding Neumman-Poincar\'e operator in solving the underlying elastostatic system in $\mathbb{R}^2$; see Remark~\ref{rem:NP} for more relevant discussion. Indeed, this is the main tool in deriving the critical plasmon constant in most of the literature mentioned earlier in the optical case. However, the spectral properties of the corresponding Neumann-Poincar\'e operator for the elastostatic system in $\mathbb{R}^3$ are not yet known. In Section~\ref{sect:3}, based on certain purely analysis means, we derive all the possible plasmon constants as well as the associated infinite dimensional kernel of the PDO $\mathcal{L}_{\widetilde\lambda,\widetilde\mu}$ in the limiting case with $\delta=0$. This also cats light on the corresponding investigation of the spectral properties of the Neumann-Poincar\'e operator for the elastostatic system in $\mathbb{R}^3$.

\item For a very broad class of forcing terms, it is shown that resonance does not occur for the plasmonic device \eqref{eq:tensor1}--\eqref{eq:tensor2} with a fixed constant $c$, which includes the one considered in \cite{AKKY,LiLiu2d} for the elastostatic system in $\mathbb{R}^2$. That is, the plasmonic device which induces resonance in $\mathbb{R}^2$ does not induce resonance in $\mathbb{R}^3$.

\item By properly choosing the plasmon constant $c$, varying according to the forcing source or the loss parameter $\delta$, we construct two novel plasmon devices, one with a core and the other without a core. If there is no core, we show that resonance always occurs, whereas if the core is nonempty and of an arbitrary shape, we show that there exists a critical radius such that resonance occurs when the source lies within the critical radius, whereas resonance does not occur when the source lies outside the critical radius. Our argument follows the general variational strategy as that in \cite{LiLiu2d}, we can only show the resonance results, and we cannot deal with the localized and cloaking effects.

\end{enumerate}

Finally, we mention in passing that the existence of exotic elastic materials with negative stiffness was reported in the physical literature; see \cite{KM} and \cite{LLBW}.

The rest of the paper is organized as follows. In Section 2, we provide some preliminary knowledge including the variational principles and the spherical harmonic representations for the elastostatic system. In Section 3, we derive the plasmon constants and construct the perfect plasmon elastic waves. Section 4 is devoted to the resonance and non-resonance results. The paper is concluded in Section 5.

\section{Preliminaries for the elastostatic system}

In this section, we collect some preliminary knowledge for the elastostatic system \eqref{eq:lame1}, including the variational principles and the spherical harmonic representations. Throughout the paper, we assume that the force term $\mathbf{f}=({f}_i)_{i=1}^3\in H^{-1}(\mathbb{R}^3)^3$ in \eqref{eq:lame1} with a compact support and a zero average in the sense that
\begin{equation}\label{eq:average}
\langle {f}_i, \mathbf{1}\rangle=0,\quad i=1,2,3,
\end{equation}
where $\mathbf{1}: \mathbb{R}^3\rightarrow\mathbb{R}$ is the constant function, $\mathbf{1}(x)=1$ for all $x\in\mathbb{R}^3$. In what follows, we let $B_R$ with $R\in\mathbb{R}_+$ denote a central ball of radius $R$ in $\mathbb{R}^3$. Without loss of generality, we assume that there exists $R_0\in\mathbb{R}_+$ such that $\mathrm{supp}(\mathbf{f})\subset B_{R_0}$. In the subsequent study, we shall also need the following Banach space
\begin{equation}\label{eq:weak2}
\mathcal{S}:=\big\{ \bu\in H_{\text{loc}}^1(\mathbb{R}^3)^3;\ \nabla\bu\in L^2(\mathbb{R}^3)^{3\times 3}\ \ \mbox{and}\ \ \int_{B_{R_0}}\bu=0 \big\},
\end{equation}
endowed with the Sobolev norm for $\mathbf{u}=(u_i)_{i=1}^N$,
\begin{equation}\label{eq:norms}
\|\mathbf{u}\|_{\mathcal{S}}:=\left(\int_{\mathbb{R}^N}\sum_{i=1}^N \|\nabla u_i\|^2\, dV+\int_{B_{R_0}}\|\mathbf{u}\|^2\, dV\right)^{1/2}.
\end{equation}

\subsection{Variational principles}
For self-containedness, we present the primal and dual variational principles for the elastostatic system \eqref{eq:lame1}, which were established in \cite{LiLiu2d} and shall play a critical role in our subsequent plasmon resonance study. For a fixed force term $\mathbf{f}\in H^{-1}(\mathbb{R}^3)^3$ and for the solution $\mathbf{u}_\delta\in H^1_{\text{loc}}(\mathbb{R}^3)^3: \mathbb{R}^3\rightarrow\mathbb{C}^3$ in \eqref{eq:lame1}, we set
\begin{equation}\label{eq:decomp1}
\mathbf{u}_\delta=\mathbf{v}_\delta+\mathrm{i}\frac{1}{\delta}\mathbf{w}_\delta,
\end{equation}
where $\mathbf{v}_\delta, \mathbf{w}_\delta\in H^1_{\text{loc}}(\mathbb{R}^3)^3:\mathbb{R}^3\rightarrow\mathbb{R}^3$ satisfying $\mathbf{v}_\delta=\mathcal{O}(\|x\|^{-1})$ and $\mathbf{w}_\delta=\mathcal{O}(\|x\|^{-1})$ as $\|x\|\rightarrow+\infty$. One has
\begin{align}
& \mathcal{L}_{\lambda_A,\mu_A}\mathbf{v}_\delta-\mathcal{L}_{\lambda,\mu}\mathbf{w}_\delta=\mathbf{f},\label{eq:couple1}\\
& \mathcal{L}_{\lambda_A,\mu_A}\mathbf{w}_\delta+\delta^2\mathcal{L}_{\lambda,\mu}\mathbf{v}_\delta=\mathbf{0},\label{eq:couple2}
\end{align}
where
\begin{equation}
  (\lambda_A(x),\mu_A(x)):=A(x)(\lambda,\mu), \quad x\in\mathbb{R}^3
\end{equation}
with $A$ is given in \eqref{eq:tensor2}, and $(\lambda, \mu)$ are the two regular Lam\'e constants in \eqref{eq:tensor1}. Furthermore, there holds
\begin{equation}\label{eq:energy3}
\begin{split}
\mathbf{E}_\delta(\mathbf{u}_\delta):=&\mathbf{E}_\delta(\mathbf{C}_{\widetilde{\lambda},\widetilde{\mu}},\mathbf{f})=\frac{\delta}{2}\mathbf{P}_{\lambda,\mu}(\mathbf{u}_\delta,\mathbf{u}_\delta)\\
=&\frac{\delta}{2}\mathbf{P}_{\lambda,\mu}(\mathbf{v}_\delta,\mathbf{v}_\delta)+\frac{1}{2\delta}\mathbf{P}_{\lambda,\mu}(\mathbf{w}_\delta,\mathbf{w}_\delta),
\end{split}
\end{equation}
where $\mathbf{E}_\delta$ is given in \eqref{eq:energy2} and $\mathbf{P}_{\lambda,\mu}$ is given in \eqref{eq:energy1}.

Next, we introduce the following energy functionals
\begin{align}
&\label{eq:ef1}\mathbf{I}_\delta(\mathbf{v},\mathbf{w}):=\frac\delta 2 \mathbf{P}_{\lambda,\mu}(\mathbf{v},\mathbf{v})+\frac{1}{2\delta}\mathbf{P}_{\lambda,\mu}(\mathbf{w},\mathbf{w})\quad\mbox{for}\ \ \ (\mathbf{v}, \mathbf{w})\in \mathcal{S}\times\mathcal{S},\\
&\label{eq:ef2}\mathbf{J}_\delta(\mathbf{v},\bm{\psi}):=\int_{\mathbb{R}^3}\mathbf{f}\cdot \bm{\psi}-\frac{\delta}{2}\mathbf{P}_{\lambda,\mu}(\mathbf{v},\mathbf{v})-\frac{\delta}{2}\mathbf{P}_{\lambda,\mu}(\bm{\psi},\bm{\psi})\ \mbox{for} \ (\mathbf{v}, \bm{\psi})\in \mathcal{S}\times\mathcal{S}.
\end{align}
and consider the following optimization problems:
\begin{equation}\label{eq:primal}
\begin{split}
&\mbox{Minimize $\mathbf{I}_\delta(\mathbf{v},\mathbf{w})$ over all pairs $(\mathbf{v},\mathbf{w})\in \mathcal{S}\times\mathcal{S}$ }\\
&\mbox{subject to the PDE constraint } \mathcal{L}_{\lambda_A,\mu_A}\mathbf{v}-\mathcal{L}_{\lambda,\mu}\mathbf{w}=\mathbf{f};
\end{split}
\end{equation}
and
\begin{equation}\label{eq:dual}
\begin{split}
&\mbox{Maximize $\mathbf{J}_\delta(\mathbf{v},\bm{\psi})$ over all pairs $(\mathbf{v},\bm{\psi})\in \mathcal{S}\times\mathcal{S}$}\\
&\mbox{subject to the PDE constraint } \mathcal{L}_{\lambda_A,\mu_A}\bm{\psi}+\delta\mathcal{L}_{\lambda,\mu}\mathbf{v}=\mathbf{0}.
\end{split}
\end{equation}
In the sequel, we shall refer to \eqref{eq:primal} and \eqref{eq:dual}, respectively, as the primal and dual variational problems for the elastostatic system \eqref{eq:lame1}, or equivalently \eqref{eq:couple1}-\eqref{eq:couple2}.

We have the following variational principles from \cite{LiLiu2d}.

\begin{thm}\label{thm:primaldual}
There holds the primal variational principle that the problem \eqref{eq:primal} is equivalent to the elastic problem \eqref{eq:lame1} in the following sense. The infimum
\[
\inf \big\{\mathbf{I}_\delta(\mathbf{v},\mathbf{w});\, \mathcal{L}_{\lambda_A,\mu_A}\mathbf{v}-\mathcal{L}_{\lambda,\mu}\mathbf{w}=\mathbf{f} \big\}
\]
is attainable at a pair $(\mathbf{v}_\delta,\mathbf{w}_\delta)\in \mathcal{S}\times\mathcal{S}$. The minimizing pair $(\mathbf{v}_\delta,\mathbf{w}_\delta)$ verifies that the function $\mathbf{u}_\delta:=\mathbf{v}_\delta+\mathrm{i}\delta^{-1}\mathbf{w}_\delta$ is the unique solution to the elastic problem \eqref{eq:lame1} and moreover one has
\begin{equation}\label{eq:primalp2}
\mathbf{E}_\delta(\mathbf{u}_\delta)=\mathbf{I}_\delta(\mathbf{v}_\delta,\mathbf{w}_\delta).
\end{equation}

Similarly, there holds the dual variational principle that the problem \eqref{eq:dual} is equivalent to the elastic problem \eqref{eq:lame1} in the following sense. The supremum
\[
\sup \big\{\mathbf{J}_\delta(\mathbf{v},\bm{\psi}); \mathcal{L}_{\lambda_A,\mu_A}\bm{\psi}+\delta\mathcal{L}_{\lambda,\mu}\bv=\mathbf{0} \big\}
\]
is attainable at a pair $(\mathbf{v}_\delta,\bm{\psi}_\delta)\in \mathcal{S}\times\mathcal{S}$. The maximizing pair $(\mathbf{v}_\delta,\bm{\psi}_\delta)$ verifies that the function $\mathbf{u}_\delta:=\mathbf{v}_\delta+\mathrm{i}\bm{\psi}_\delta$ is the unique solution to the elastic problem \eqref{eq:lame1}, and moreover one has
\begin{equation}\label{eq:dualp2}
\mathbf{E}_\delta(\mathbf{u}_\delta)=\mathbf{J}_\delta(\mathbf{v}_\delta,\bm{\psi}_\delta).
\end{equation}
\end{thm}

\subsection{Spherical harmonic representations}

For the subsequent use, we present some results on the spherical harmonic representations to the solutions of the following Lam\'e equations
\begin{equation}\label{eq:original_1}
 \mathcal{L}_{\lambda,\mu}\bu:=\mu\mathbf{u}+(\lambda+\mu)\nabla\nabla\cdot\mathbf{u}=\mathbf{0}.
\end{equation}
We also refer to \cite{LASS} for more relevant discussion. The solution $\mathbf{u}$ to \eqref{eq:original_1} is called foregoing/outgoing if it decays as follows
\begin{equation}\label{eq:decay1}
\mathbf{u}(x)=\mathcal{O}\left( \|x\|^{-1} \right)\quad\mbox{as}\quad \|x\|\rightarrow+\infty.
\end{equation}

In the sequel, for $x\in\mathbb{R}^3\backslash\{0\}$, we shall make use the spherical coordinates
\[
x=(x_j)_{j=1}^3=r\cdot \hat{x}\quad\mbox{with}\ \ r=\|x\|\ \ \mbox{and}\ \ \hat{x}=r^{-1}\cdot x.
\]
For $n\in\mathbb{N}$, we let $Y_n^m(\hat x)(m=n,n-1,...,1,0,...,-n+1,-n)$ denote the orthonormalized Laplace spherical harmonic polynomial of degree $n$ and order $m$. In what follows, we set $\mathbf{Y}_n(\hat{x})$ to be the vector of size $2n-1$,
\begin{equation}
  \mathbf{Y}_n(\hat x)=\left[
                \begin{array}{c}
                  Y_n^n(\hat x) \\
                \vdots \\
                Y_n^{-n}(\hat x)\\
                \end{array}
              \right];
\end{equation}
and introduce the matrices $\mathbf{D}_{n+1}^{nx_j}$ and $\mathbf{D}_n^{(n+1)x_j}$, $1\leq j\leq 3$, be such that
\begin{equation}
 \frac{\partial}{\partial x_j}\left[ r^{n+1} \mathbf{Y}_{n+1}(\hat x)\right] =\mathbf{D}_{n+1}^{nx_j} r^n \mathbf{Y}_n(\hat x)
\end{equation}
and
\begin{equation}
 \frac{\partial}{\partial x_j}\left[ r^{-n-1} \mathbf{Y}_{n}(\hat x)\right] = \mathbf{D}_n^{(n+1)x_j} r^{-n-2} \mathbf{Y}_{n+1}(\hat x).
\end{equation}
We also introduce a coefficient matrix $\mathbf{G}^n$ of size $3\times (2n+1)$ as follows,
\begin{equation}
  \mathbf{G}^n=\left[
        \begin{array}{ccccc}
          a_n^n & a_n^{n-1} & \ldots & a_n^{-n+1} & a_n^{-n} \\
          b_n^n & b_n^{n-1} & \ldots & b_n^{-n+1} & b_n^{-n} \\
          c_n^n & c_n^{n-1} & \ldots & c_n^{-n+1} & c_n^{-n} \\
        \end{array}
      \right],
\end{equation}
where $a_n^j, b_n^j$ and $c_n^j$, $-n\leq j\leq n$ are all complex numbers. We use $\mathbf{G}^n_{j,:}$, $1\leq j\leq 3$, to denote the $j$-th row of the matrix $\mathbf{G}^n$; that is, e.g.,
\begin{equation}
  \mathbf{G}^n_{1,:}=\left[
              \begin{array}{ccccc}
               a_n^n & a_n^{n-1} & \ldots & a_n^{-n+1} & a_n^{-n} \\
              \end{array}
            \right].
\end{equation}

The general form of a foregoing solution to (\ref{eq:original_1}) can be written as
\begin{equation}\label{eq:outgoing1}
  \bu^o(x)=\sum_{n=1}^{\infty}
\mathbf{G}^n r^{-n-1} \mathbf{Y}_n(\hat x) +
      k_n  \left[
             \begin{array}{c}
               \mathbf{t}^1_{n+1} \mathbf{D}^{(n+2)x_1}_{n+1} \\
               \mathbf{t}^1_{n+1} \mathbf{D}^{(n+2)x_2}_{n+1}\\
               \mathbf{t}^1_{n+1} \mathbf{D}^{(n+2)x_3}_{n+1}\\
             \end{array}
           \right] r^{-n-1}\mathbf{Y}_{n+2}(\hat x),
\end{equation}
where
\begin{equation}\label{eq:k_n}
  k_n := \frac{\lambda + \mu}{2 ((n+2)\lambda + (3n+5)\mu )},
\end{equation}
and
\begin{equation}\label{eq:tn1}
  \mathbf{t}^1_{n+1} := \mathbf{G}^n_{1,:} \mathbf{D}^{(n+1)x_1}_n + \mathbf{G}^n_{2,:} \mathbf{D}^{(n+1)x_2}_n + \mathbf{G}^n_{3,:} \mathbf{D}^{(n+1)x_3}_n.
\end{equation}
The general form of an entire solution to (\ref{eq:original_1}) can be expressed as
\begin{equation}\label{eq:innergoing1}
  \bu^i(x)=\sum_{n=1}^{\infty}
 \mathbf{G}^n r^n \mathbf{Y}_n(\hat{x})- M_n
 \left[
   \begin{array}{c}
     \mathbf{t}^3_{n-1} \mathbf{D}^{(n-2)x_1}_{n-1} \\
     \mathbf{t}^3_{n-1} \mathbf{D}^{(n-2)x_2}_{n-1} \\
     \mathbf{t}^3_{n-1} \mathbf{D}^{(n-2)x_3}_{n-1} \\
   \end{array}
 \right]  r^{n}\mathbf{Y}_{n-2},
\end{equation}
where
\begin{equation}\label{eq:coefficient_N_n}
  M_n :=\frac{\lambda + \mu}{2( (n-1)\lambda +(3n-2)\mu)},
\end{equation}
and
\begin{equation}\label{eq:tn3}
  \mathbf{t}^3_{n-1} := \mathbf{G}^n_{1,:} \mathbf{D}^{(n-1)x_1}_n + \mathbf{G}^n_{2,:} \mathbf{D}^{(n-1)x_2}_n + \mathbf{G}^n_{3,:} \mathbf{D}^{(n-1)x_3}_n.
\end{equation}

For a solution $\mathbf{u}$ to \eqref{eq:original_1} inside a ball $B_R$, if the surface displacement is prescribed on $\partial B_R$, say
\begin{equation}\label{eq:sfd1}
\mathbf{u}(R\hat x)=
\left[
  \begin{array}{c}
    \mathbf{A}_n \\
    \mathbf{B}_n \\
    \mathbf{C}_n \\
  \end{array}
\right]
 \mathbf{Y}_n(\hat x),
\end{equation}
where $\mathbf{A}_n, \mathbf{B}_n, \mathbf{C}_n$ are the coefficient vectors of size $2n+1$, then by straightforward (though a bit tedious) calculations, one has that
\begin{equation}\label{eq:solution_inside}
  \bu(x)=\sum_{n=1}^{\infty}
\left[
  \begin{array}{c}
    \mathbf{A}_n \\
    \mathbf{B}_n \\
    \mathbf{C}_n \\
  \end{array}
\right] \frac{r^n}{R^n} \mathbf{Y}_n + M_{n+2}\frac{R^2-r^2}{R^{n+2}}
\left[
  \begin{array}{c}
    \mathbf{t}^2_{n+1} \mathbf{D}^{nx_1}_{n+1} \\
    \mathbf{t}^2_{n+1} \mathbf{D}^{nx_2}_{n+1} \\
    \mathbf{t}^2_{n+1} \mathbf{D}^{nx_3}_{n+1} \\
  \end{array}
\right] r^{n} \mathbf{Y}_{n}(\hat{x}),
\end{equation}
where
\begin{equation}\label{eq:Mn}
  M_{n+2} :=\frac{1}{2} \frac{{\lambda} + {\mu}}{ (n+1){\lambda} + (3n+4) {\mu} },
\end{equation}
and
\begin{equation}\label{eq:tn2}
  \mathbf{t}^2_{n+1} :=\mathbf{A}_{n+2} \mathbf{D}^{(n+1)x_1}_{n+2} + \mathbf{B}_{n+2} \mathbf{D}^{(n+1)x_2}_{n+2}  +\mathbf{C}_{n+2} \mathbf{D}^{(n+1)x_3}_{n+2}.
\end{equation}
On the other hand,  if the surface traction is prescribed on $\partial B_R$, say
\begin{equation}\label{eq:sfd2}
\frac{\partial\mathbf{u}}{\partial\bm{\nu}_{\lambda,\mu}}(R\hat x)=
\left[
  \begin{array}{c}
    \mathbf{A}_n' \\
    \mathbf{B}_n' \\
    \mathbf{C}_n' \\
  \end{array}
\right] \mathbf{Y}_n(\hat x),
\end{equation}
where $\mathbf{A}_n', \mathbf{B}_n', \mathbf{C}_n'$ are the coefficient vectors of size $2n+1$, then by straightforward calculations, one can show that the solution $\mathbf{u}$ inside $B_R$ is still given by \eqref{eq:solution_inside}, but with the coefficients
$(\mathbf{A}_n, \mathbf{B}_n, \mathbf{C}_n)$ replaced by $(\widetilde{\mathbf{A}}_n, \widetilde{\mathbf{B}}_n, \widetilde{\mathbf{C}}_n )$ as follows,
\begin{equation}\label{eq:st1}
\begin{split}
\widetilde{\mathbf{A}}_n =& \frac{R}{(n-1) {\mu}}(\mathbf{A}_n^{'} + s^1_n \mathbf{t}^4_n \mathbf{D}^{nx_1}_{n-1} + s^2_n \mathbf{t}^5_n \mathbf{D}^{nx_1}_{n+1} ), \\
\widetilde{\mathbf{B}}_n =& \frac{R}{(n-1) {\mu}}(\mathbf{B}_n^{'} + s^1_n \mathbf{t}^4_n \mathbf{D}^{nx_2}_{n-1} + s^2_n \mathbf{t}^5_n \mathbf{D}^{nx_2}_{n+1} ), \\
\widetilde{\mathbf{C}}_n =& \frac{R}{(n-1) {\mu}}(\mathbf{C}_n^{'} + s^1_n \mathbf{t}^4_n \mathbf{D}^{nx_3}_{n-1} + s^2_n \mathbf{t}^5_n \mathbf{D}^{nx_3}_{n+1} ),
\end{split}
\end{equation}
where
\begin{equation}\label{eq:tn45}
\begin{split}
\mathbf{t}^4_n:=& \mathbf{A}_n^{'} \mathbf{D}^{(n-1)x_1}_{n} + \mathbf{B}_n^{'} \mathbf{D}^{(n-1)x_2}_{n} + \mathbf{C}_n^{'} \mathbf{D}^{(n-1)x_3}_{n},\\
\mathbf{t}^5_n:=& \mathbf{A}_n^{'} \mathbf{D}^{(n+1)x_1}_{n} + \mathbf{B}_n^{'} \mathbf{D}^{(n+1)x_2}_{n} + \mathbf{C}_n^{'} \mathbf{D}^{(n+1)x_3}_{n},
\end{split}
\end{equation}
and
\begin{equation}
\begin{split}
s^1_n:=& \frac{E_n}{n-1+n(2n+1)E_n},\\
s^2_n:=& \frac{1}{2n(2n+1)},
\end{split}
\end{equation}
with
\begin{equation}\label{eq:en}
 E_n:=\frac{1}{2n+1} \frac{(n+2){\lambda} - (n-3){\mu} }{(n-1){\lambda} + (3n-2) {\mu}}.
\end{equation}

\section{Perfect plasmon elastic waves}\label{sect:3}

Let us consider the following elastostatic system for
$\bm{\psi}\in H_{loc}^1(\mathbb{R}^3)^3: \mathbb{R}^3 \rightarrow \mathbb{C}^3$
\begin{equation}\label{eq:trial1}
\begin{cases}
  & \mathcal{L}_{\lambda_A,\mu_A}\bm{\psi}=\mathbf{0}, \medskip\\
   & \bm{\psi}|_-=\bm{\psi}|_+,\quad \partial_{\bm{\nu}_{\lambda_A,\mu_A}} \bm{\psi}|_-=\partial_{\bm{\nu}_{\lambda_A,\mu_A}} \bm{\psi}|_+\quad\mbox{on}\ \ \partial B_R,\medskip\\
  & \bm{\psi}(x)= \mathcal{O}(\|x\|^{-1}) \quad \mbox{as} \quad \|x\| \rightarrow \infty,
\end{cases}
\end{equation}
where
\begin{equation}\label{eq:A2}
   A(x) = \begin{cases}
   c,\quad & \|x\| \leq R,\\
   +1,\quad & \|x\| >R.
   \end{cases}
\end{equation}
Clearly, if $c$ is a positive constant, then by the well-posedness of the elastostatic system \eqref{eq:trial1}, one must have that $\bm{\psi}\equiv \bm{0}$. We seek nontrivial solutions to \eqref{eq:trial1} when $c$ is allowed to be negative-valued. Those nontrivial solutions are referred to as the \emph{perfect plasmon elastic waves}, and in combination with the variational principles in Theorem~\ref{thm:primaldual}, they shall be crucial for our subsequent establishment of the plasmon resonances for \eqref{eq:tensor1}--\eqref{eq:lame1} in $\mathbb{R}^3$. Indeed, we have
\begin{thm}\label{thm:perfectwaves}
Consider the PDE system \eqref{eq:trial1}--\eqref{eq:A2} for a function $\bm{\psi}\in H^1_{loc}(\mathbb{R}^3)^3: \mathbb{R}^3 \rightarrow \mathbb{R}^3$.
Let $n\in\mathbb{N}$ and $n\geq 2$ be fixed and set
\begin{equation}\label{eq:coefficient_pi}
  \begin{split}
    \zeta_1 &: = -1-\frac{3}{n-1}, \\
    \zeta_2 & := -\frac{(2n+2)((n-1) \lambda + (2n-2) \mu)}{(2n^2 + 1) \lambda + (2 + 2n(n-1))\mu}, \\
    \zeta_3 & := -\frac{(2n^2 + 4n + 3)\lambda + (2n^2 + 6n +6)\mu}{2n((n+2)\lambda + (3n + 5)\mu)}.
  \end{split}
\end{equation}
 Then if
\begin{equation}\label{eq:tensor5}
 c=\zeta_1,
\end{equation}
there exists a non-trivial solution $\bm{\psi} = \widehat{\bm{\psi}}_{n,k}\in H_{loc}^1(\mathbb{R}^3)^3$ as follows:
\begin{equation}\label{eq:solution1}
\widehat{\bm{\psi}}_{n,k}(x)=\begin{cases}
  \mathbf{G}^{n,\zeta_1,k}  r^n \mathbf{Y}_n(\hat{x}),               & r \leq R, \\
  \mathbf{G}^{n,\zeta_1,k} \frac{R^{2n+1}}{ r^{n+1} } \mathbf{Y}_n(\hat{x}), & r>R,
\end{cases}
\end{equation}
with $\mathbf{G}^{n,\zeta_1,k}$, $k=1,2,\ldots,2n+1$, satisfying
\begin{equation}\label{eq:solution12}
  \mathbf{t}^1_{n+1}=\mathbf{0} \quad \mbox{and} \quad \mathbf{t}^3_{n-1}=\mathbf{0},
\end{equation}
where $\mathbf{t}^1_{n+1}$ and $\mathbf{t}^3_{n-1}$ are, respectively, defined in \eqref{eq:tn1} and \eqref{eq:tn3} with $\mathbf{G}^n_{i,:}$ replaced by $\mathbf{G}^{n,\zeta_1,k}_{i,:}$, $i=1,2,3$.

If
\begin{equation}
  c = \zeta_2,
\end{equation}
there exists a non-trivial solution $\bm{\psi} = \widehat{\bm{\psi}}_{n,k}\in H_{loc}^1(\mathbb{R}^3)^3$ as follows:
\begin{equation}\label{eq:solution2}
\widehat{\bm{\psi}}_{n,k}(x)=\begin{cases}
   \mathbf{G}^{n,\zeta_2,k}  r^n \mathbf{Y}_n(\hat{x}) - M_n(r^2 - R^2) \left[
                                                 \begin{array}{c}
                                                   \mathbf{t}^3_{n-1} \mathbf{D}^{(n-2)x_1}_{n-1} \\
                                                   \mathbf{t}^3_{n-1} \mathbf{D}^{(n-2)x_2}_{n-1}\\
                                                   \mathbf{t}^3_{n-1} \mathbf{D}^{(n-2)x_3}_{n-1}\\
                                                 \end{array}
                                               \right] r^{n-2} \mathbf{Y}_{n-2}(\hat{x})
 ,               & r \leq R,\\
  \mathbf{G}^{n,\zeta_2,k} \frac{R^{2n+1}}{ r^{n+1} } \mathbf{Y}_n(\hat{x}), & r>R,
\end{cases}
\end{equation}
with $\mathbf{G}^{n,\zeta_2,k}$, $k=1,2,\ldots,2n-1$, satisfying
\begin{equation}\label{eq:solution22}
  \mathbf{t}^1_{n+1}=\mathbf{0} \quad \mbox{and} \quad \mathbf{t}^3_{n-1}\neq \mathbf{0},
\end{equation}
where $\mathbf{t}^1_{n+1}$ and $\mathbf{t}^3_{n-1}$ are, respectively, defined in \eqref{eq:tn1} and \eqref{eq:tn3} with $\mathbf{G}^n_{i,:}$ replaced by $\mathbf{G}^{n,\zeta_2,k}_{i,:}$, $i=1,2,3$.

If
\begin{equation}
  c = \zeta_3,
\end{equation}
there exists a non-trivial solution $\bm{\psi} = \widehat{\bm{\psi}}_{n,k}\in H_{loc}^1(\mathbb{R}^3)^3$ as follows:
\begin{equation}\label{eq:solution3}
\widehat{\bm{\psi}}_{n,k}(x)=\begin{cases}
  \mathbf{G}^{n,\zeta_3,k}  r^n \mathbf{Y}_n(\hat{x}) ,               & r \leq R,\\
  \mathbf{G}^{n,\zeta_3,k} \frac{R^{2n+1}}{ r^{n+1} } \mathbf{Y}_n(\hat{x}) + k_n(r^2 - R^2) \left[
                                                 \begin{array}{c}
                                                   \mathbf{t}^1_{n+1} \mathbf{D}^{(n+2)x_1}_{n+1} \\
                                                   \mathbf{t}^1_{n+1} \mathbf{D}^{(n+2)x_2}_{n+1}\\
                                                   \mathbf{t}^1_{n+1} \mathbf{D}^{(n+2)x_3}_{n+1}\\
                                                 \end{array}
                                               \right]  \frac{R^{2n+1}}{ r^{n+3} } \mathbf{Y}_{n+2}(\hat{x})
, & r>R,
\end{cases}
\end{equation}
with $\mathbf{G}^{n,\zeta_3,k}$, $k=1,2,\ldots,2n+3$,  satisfying
\begin{equation}\label{eq:solution32}
  \mathbf{t}^1_{n+1}\neq \mathbf{0} \quad \mbox{and} \quad \mathbf{t}^3_{n-1}=\mathbf{0},
\end{equation}
where $\mathbf{t}^1_{n+1}$ and $\mathbf{t}^3_{n-1}$ are, respectively, defined in \eqref{eq:tn1} and \eqref{eq:tn3} with $\mathbf{G}^n_{i,:}$ replaced by $\mathbf{G}^{n,\zeta_3,k}_{i,:}$, $i=1,2,3$.

\end{thm}

By Theorem~\ref{thm:perfectwaves}, we find a series of perfect plasmon waves, $\widehat{\bm{\psi}}_{n,k}$, $n=2,3,4,\ldots$, and these are actually resonant modes for our subsequent use in Section~\ref{sect:4}. The form of the plasmon constant $c$ in \eqref{eq:tensor5} is critical for the existence of the perfect plasmon waves \eqref{eq:solution1}. With the explicit forms of the plasmon constant in \eqref{eq:tensor5} and the perfect plasmon wave in \eqref{eq:solution1}, one can verify by direct calculations that $\widehat{\bm{\psi}}_n\in H_{loc}^1(\mathbb{R}^3)^3$ and satisfies the PDE system \eqref{eq:trial1}--\eqref{eq:A2}. Nevertheless, in what follows, we shall give a proof of Theorem~\ref{thm:perfectwaves}, starting from the most general construction of perfect plasmon waves for the PDE system \eqref{eq:trial1}--\eqref{eq:A2}. The general idea is that we represent the solution inside $B_R$ and outside $B_R$, respectively, by the spherical harmonic expansions \eqref{eq:innergoing1} and \eqref{eq:outgoing1}. Then with the help of our discussion in \eqref{eq:sfd1}--\eqref{eq:en}, we match the surface displacement and traction of the elastic field on the sphere $\partial B_R$. Throughout this trial process, we leave the plasmon constant $c$ as a free parameter. It turns out that \eqref{eq:coefficient_pi} gives all the possibilities that one can determine nontrivial solutions of the forms, \eqref{eq:solution1}-\eqref{eq:solution12}, \eqref{eq:solution2}-\eqref{eq:solution22} and \eqref{eq:solution3}-\eqref{eq:solution32}, respectively. That is, we actually have found all the possible plasmon resonances in $\mathbb{R}^3$ for the elastostatic system \eqref{eq:tensor1}--\eqref{eq:lame1}.

\begin{proof}[Proof of Theorem~\ref{thm:perfectwaves}]
First, by \eqref{eq:outgoing1}, we represent the solution $\bm{\psi}$ outside $B_R$ by the following harmonic expansion,
\begin{equation}\label{solution_outside}
  \left[
    \begin{array}{c}
      u^o \\
      v^o \\
      w^o \\
    \end{array}
  \right]= \sum_{n=1}^{\infty}
  \mathbf{G}^n r^{-n-1} \mathbf{Y}_n(\hat x) +
      k_n  \left[
             \begin{array}{c}
               \mathbf{t}^1_{n+1} \mathbf{D}^{(n+2)x_1}_{n+1} \\
               \mathbf{t}^1_{n+1} \mathbf{D}^{(n+2)x_2}_{n+1}\\
               \mathbf{t}^1_{n+1} \mathbf{D}^{(n+2)x_3}_{n+1}\\
             \end{array}
           \right] r^{-n-1}\mathbf{Y}_{n+2}(\hat x),
\end{equation}
where $\mathbf{G}^n$ is coefficient matrix of size $3\times 2n+1$, and $k_n, \mathbf{t}^1_{n+1}$ are given in (\ref{eq:k_n}), (\ref{eq:tn1}) respectively.
Here and also in what follows, if the superscript of $\mathbf{G}^n$ is nonpositive, we set the matrix to be identically zero.

Then, by \eqref{solution_outside}, one readily has that the displacement of the solution $\bm{\psi}$ on $\partial B_R$ is given by
\begin{equation}
  \bm{\psi}(R\hat x)=
  \left[
    \begin{array}{c}
      \mathbf{A}_n \\
      \mathbf{B}_n \\
      \mathbf{C}_n \\
    \end{array}
  \right] \mathbf{Y}_n(\hat x)
\end{equation}
with
\begin{equation}\label{coefficient_solution_inside_1}
\left\{
  \begin{array}{ll}
    \mathbf{A}_n= & \mathbf{G}^n_{1,;}/R^{n+1} + k_{n-2} \mathbf{t}^1_{n-1} \mathbf{D}^{nx_1}_{n-1}/R^{n-1},\medskip \\
    \mathbf{B}_n= & \mathbf{G}^n_{2,;}/R^{n+1} + k_{n-2} \mathbf{t}^1_{n-1} \mathbf{D}^{nx_2}_{n-1}/R^{n-1},\medskip \\
    \mathbf{C}_n= & \mathbf{G}^n_{3,;}/R^{n+1} + k_{n-2} \mathbf{t}^1_{n-1} \mathbf{D}^{nx_3}_{n-1}/R^{n-1}.
  \end{array}
\right.
\end{equation}

By using our argument in \eqref{eq:sfd1}--\eqref{eq:tn2}, one has from \eqref{coefficient_solution_inside_1} that the solution $\bm{\psi}$ inside $B_R$ is given by
\begin{equation}\label{solution_inside_1}
  \left[
    \begin{array}{c}
      u^i \\
      v^i \\
      w^i \\
    \end{array}
  \right]= \sum_{n=1}^{\infty}
\left[
  \begin{array}{c}
    \mathbf{A}_n \\
    \mathbf{B}_n \\
    \mathbf{C}_n \\
  \end{array}
\right] \frac{r^n}{R^n} \mathbf{Y}_n(\hat x) + M_{n+2}\frac{R^2-r^2}{R^{n+2}}
\left[
  \begin{array}{c}
    \mathbf{t}^2_{n+1} \mathbf{D}^{nx_1}_{n+1} \\
    \mathbf{t}^2_{n+1} \mathbf{D}^{nx_2}_{n+1} \\
    \mathbf{t}^2_{n+1} \mathbf{D}^{nx_3}_{n+1} \\
  \end{array}
\right] r^{n} \mathbf{Y}_{n}(\hat{x}),
\end{equation}
where $M_{n+2}$ and $\mathbf{t}^2_{n+1}$ are given in (\ref{eq:Mn}) and (\ref{eq:tn2}) respectively.

Next we consider the traction on the sphere $\partial B_R$. Using \eqref{solution_outside} and by direct calculation, the traction on the sphere $\partial B_R$ of the solution $\bm{\psi}$ outside the sphere $B_R$ has the following representation:
\begin{equation}
 \frac{\partial\bm{\psi}}{\partial\bm{\nu}_{\lambda_A,\mu_A}}(R\hat x)= \sum_{n=1}^{\infty}
 \left[
   \begin{array}{c}
     \mathbf{A}_n^{'} \\
     \mathbf{B}_n^{'} \\
     \mathbf{C}_n^{'} \\
   \end{array}
 \right] \mathbf{Y}_n(\hat{x}),
\end{equation}
where
\begin{equation}\label{coefficient_solution_inside_2}
\left\{
  \begin{array}{ll}
    \mathbf{A}_n^{'}= & (l_n \mathbf{t}^1_{n+1} \mathbf{D}^{nx_1}_{n+1} + m_n R^2 \mathbf{t}^1_{n-1} \mathbf{D}^{nx_1}_{n-1} + \frac{1}{2n+1} \mathbf{t}^3_{n-1} \mathbf{D}^{nx_1}_{n-1} + \mathbf{G}_{1,:}(-n-2)) \frac{\mu}{R^{n+2}}  ,\medskip \\
    \mathbf{B}_n^{'}= & (l_n \mathbf{t}^1_{n+1} \mathbf{D}^{nx_2}_{n+1} + m_n R^2 \mathbf{t}^1_{n-1} \mathbf{D}^{nx_2}_{n-1} + \frac{1}{2n+1} \mathbf{t}^3_{n-1} \mathbf{D}^{nx_2}_{n-1} + \mathbf{G}_{2,:}(-n-2)) \frac{\mu}{R^{n+2}}  ,\medskip \\
    \mathbf{C}_n^{'}= & (l_n \mathbf{t}^1_{n+1} \mathbf{D}^{nx_3}_{n+1} + m_n R^2 \mathbf{t}^1_{n-1} \mathbf{D}^{nx_3}_{n-1} + \frac{1}{2n+1} \mathbf{t}^3_{n-1} \mathbf{D}^{nx_3}_{n-1} + \mathbf{G}_{3,:}(-n-2)) \frac{\mu}{R^{n+2}} ,
  \end{array}
\right.
\end{equation}
with
\begin{align}
  l_n=& \left(\frac{2\lambda}{\lambda + \mu} + \frac{2(-n-2)}{2n+3}\right)k_n - \frac{2}{(2n+3)(2n+1)},\\
  m_n=& \left( \frac{-2\lambda}{\lambda + \mu} - \frac{4n(n-1)}{2n-1} \right) k_{n-2} - \frac{1}{2n-1},\\
  \mathbf{t}^3_{n-1} =& \mathbf{G}^n_{1,:} \mathbf{D}^{(n-1)x_1}_{n} + \mathbf{G}^n_{2,:} \mathbf{D}^{(n-1)x_2}_{n} + \mathbf{G}^n_{3,:} \mathbf{D}^{(n-1)x_3}_{n}.
\end{align}
By the transmission condition, the traction of the solution $\bm{\psi}$ inside the ball $B_R$ should also be given by
\begin{equation}
\sum_{n=1}^{\infty}(\mathbf{A}_n^{'}, \mathbf{B}_n^{'}, \mathbf{C}_n^{'})\mathbf{Y}_n(\hat{x}),
\end{equation}
which, with the help of our argument in \eqref{eq:sfd2}--\eqref{eq:en} and along with straightforward (though lengthy and tedious) calculations, can determine the solution $\bm{\psi}$ inside $B_R$ of the same form as (\ref{solution_inside_1}), but with the coefficients $(\mathbf{A}_n, \mathbf{B}_n, \mathbf{C}_n)$, replaced by $(\widetilde{\mathbf{A}}_n, \widetilde{\mathbf{B}}_n, \widetilde{\mathbf{C}}_n )$, $n\geq 2$, as
\begin{equation}\label{coefficient_solution_inside_3}
  \left\{
    \begin{array}{ll}
      \widetilde{\mathbf{A}}_n = \frac{R}{(n-1) c \mu}(\mathbf{A}_n^{'} + s^1_n \mathbf{t}^4_n \mathbf{D}^{nx_1}_{n-1} + s^2_n \mathbf{t}^5_n \mathbf{D}^{nx_1}_{n+1} ),\medskip \\
      \widetilde{\mathbf{B}}_n = \frac{R}{(n-1) c \mu}(\mathbf{B}_n^{'} + s^1_n \mathbf{t}^4_n \mathbf{D}^{nx_2}_{n-1} + s^2_n \mathbf{t}^5_n \mathbf{D}^{nx_2}_{n+1} ),\medskip \\
      \widetilde{\mathbf{C}}_n = \frac{R}{(n-1) c \mu}(\mathbf{C}_n^{'} + s^1_n \mathbf{t}^4_n \mathbf{D}^{nx_3}_{n-1} + s^2_n \mathbf{t}^5_n \mathbf{D}^{nx_3}_{n+1} ),
    \end{array}
  \right.
\end{equation}
where
\begin{align}
  \mathbf{t}^4_n=& \mathbf{A}_n^{'} \mathbf{D}^{(n-1)x_1}_{n} + \mathbf{B}_n^{'} \mathbf{D}^{(n-1)x_2}_{n} + \mathbf{C}_n^{'} \mathbf{D}^{(n-1)x_3}_{n},\\
  \mathbf{t}^5_n=& \mathbf{A}_n^{'} \mathbf{D}^{(n+1)x_1}_{n} + \mathbf{B}_n^{'} \mathbf{D}^{(n+1)x_2}_{n} + \mathbf{C}_n^{'} \mathbf{D}^{(n+1)x_3}_{n},\\
  s^1_n=& \frac{E_n}{n-1+n(2n+1)E_n},\\
  s^2_n=& \frac{1}{2n(2n+1)},
\end{align}
and
\begin{equation}
  E_n=\frac{1}{2n+1} \frac{(n+2) {\lambda} - (n-3){\mu} }{(n-1){\lambda} + (3n-2) {\mu}}.
\end{equation}
Particularly, we note here that when $n=1$, $(\widetilde{\mathbf{A}}_1, \widetilde{\mathbf{B}}_1, \widetilde{\mathbf{C}}_1 )$ is the solution to the following system of equations:
\begin{equation}\label{eq:pe1}
    \begin{cases}
      -E_1 ( \widetilde{\mathbf{A}}_1 \mathbf{D}^{0x_1}_1+& \hspace*{-3mm}\widetilde{\mathbf{B}}_1 \mathbf{D}^{0x_2}_1  + \widetilde{\mathbf{C}}_1  \mathbf{D}^{0x_3}_1)  \mathbf{D}^{1x_1}_0\\
      & - (\widetilde{\mathbf{A}}_1 \mathbf{D}^{2x_1}_1+ \widetilde{\mathbf{B}}_1 \mathbf{D}^{2x_2}_1  + \widetilde{\mathbf{C}}_1  \mathbf{D}^{2x_3}_1)\mathbf{D}^{1x_1}_2/3=R \mathbf{A}_1^{'}/(c \mu)\medskip \\
      -E_1 ( \widetilde{\mathbf{A}}_1 \mathbf{D}^{0x_1}_1+& \hspace*{-3mm}\widetilde{\mathbf{B}}_1 \mathbf{D}^{0x_2}_1  + \widetilde{\mathbf{C}}_1  \mathbf{D}^{0x_3}_1)  \mathbf{D}^{1x_2}_0 \\
      &- (\widetilde{\mathbf{A}}_1 \mathbf{D}^{2x_1}_1+ \widetilde{\mathbf{B}}_1 \mathbf{D}^{2x_2}_1  + \widetilde{\mathbf{C}}_1  \mathbf{D}^{2x_3}_1)\mathbf{D}^{1x_2}_2/3=R \mathbf{B}_1^{'}/(c \mu)\medskip \\
      -E_1 ( \widetilde{\mathbf{A}}_1 \mathbf{D}^{0x_1}_1+& \hspace*{-3mm}\widetilde{\mathbf{B}}_1 \mathbf{D}^{0x_2}_1  +\widetilde{\mathbf{C}}_1  \mathbf{D}^{0x_3}_1)  \mathbf{D}^{1x_3}_0\\
      & - (\widetilde{\mathbf{A}}_1 \mathbf{D}^{2x_1}_1+ \widetilde{\mathbf{B}}_1 \mathbf{D}^{2x_2}_1  + \widetilde{\mathbf{C}}_1  \mathbf{D}^{2x_3}_1)\mathbf{D}^{1x_3}_2/3=R \mathbf{C}_1^{'}/(c \mu).
    \end{cases}
\end{equation}
The equation (\ref{eq:pe1}) useless in our subsequent discussion, and we simply set the coefficient matrix to be identically zero, namely, $\mathbf{G}^1=\mathbf{0}$. Next, we consider the situation when $n\geq 2$. Through our earlier arguments, we have two expressions of the solution $\bm{\psi}$ inside $B_R$. For the consistence, one then has
\begin{equation}\label{eq:coefficient_1}
  \mathbf{A}_n=\widetilde{\mathbf{A}}_n \quad  \mathbf{B}_n=\widetilde{\mathbf{B}}_n \quad  \mathbf{C}_n=\widetilde{\mathbf{C}}_n, \quad n\geq 2,
\end{equation}
where $(\mathbf{A}_n, \mathbf{B}_n, \mathbf{C}_n)$ and $(\widetilde{\mathbf{A}}_n, \widetilde{\mathbf{B}}_n, \widetilde{\mathbf{C}}_n )$ are given in (\ref{coefficient_solution_inside_1}) and (\ref{coefficient_solution_inside_3}), respectively.

We first have from $\mathbf{A}_n=\widetilde{\mathbf{A}}_n$ in \eqref{eq:coefficient_1} that:
\begin{equation}\label{eq:coefficient_1_same}
  \begin{split}
    & \left( \mathbf{G}^n_{1,:} + \frac{\mathbf{G}^n_{1,;}(n+2)}{c(n-1)}\right)\frac{1}{R^{n+1}} +k_{n-2} \mathbf{t}^1_{n-1} \mathbf{D}^{nx_1}_{n-1}/R^{n-1}  \\
    & -\frac{1}{c(n-1)R^{n+1}} \bigg( m_n R^2 \mathbf{t}^1_{n-1}\big(\mathbf{D}^{nx_1}_{n-1} + s^1_n \mathbf{s}^4_n \mathbf{D}^{nx_1}_{n-1} + s^2_n \mathbf{s}^5_n \mathbf{D}^{nx_1}_{n+1}\big)  \\
    & \mathbf{t}^1_{n+1}\big((l_n - (n+2)s^2_n)\mathbf{D}^{nx_1}_{n+1} + l_n s^2_n \mathbf{s}^{6}_n \mathbf{D}^{nx_1}_{n+1} + l_n s^1_n \mathbf{s}^{3}_n \mathbf{D}^{nx_1}_{n-1}\big)   \\
    & \frac{\mathbf{t}^3_{n-1}}{2n+1}\big((1 + s^1_n \mathbf{s}^4_n - (2n+1)(n+2)s^1_n )\mathbf{D}^{nx_1}_{n-1} + s^2_n \mathbf{s}^5_n \mathbf{D}^{nx_1}_{n+1}\big) \bigg),
  \end{split}
\end{equation}
where
\begin{equation}\label{eq:dr1}
\begin{split}
  \mathbf{s}^3_n=& \mathbf{D}^{nx_1}_{n+1} \mathbf{D}^{(n-1)x_1}_{n} + \mathbf{D}^{nx_2}_{n+1} \mathbf{D}^{(n-1)x_2}_{n} + \mathbf{D}^{nx_3}_{n+1} \mathbf{D}^{(n-1)x_3}_{n},\\
  \mathbf{s}^4_n=& \mathbf{D}^{nx_1}_{n-1} \mathbf{D}^{(n-1)x_1}_{n} + \mathbf{D}^{nx_2}_{n-1} \mathbf{D}^{(n-1)x_2}_{n} + \mathbf{D}^{nx_3}_{n-1} \mathbf{D}^{(n-1)x_3}_{n},\\
  \mathbf{s}^5_n=& \mathbf{D}^{nx_1}_{n-1} \mathbf{D}^{(n+1)x_1}_{n} + \mathbf{D}^{nx_2}_{n-1} \mathbf{D}^{(n+1)x_2}_{n} + \mathbf{D}^{nx_3}_{n-1} \mathbf{D}^{(n+1)x_3}_{n},\\
  \mathbf{s}^6_n=& \mathbf{D}^{nx_1}_{n+1} \mathbf{D}^{(n+1)x_1}_{n} + \mathbf{D}^{nx_2}_{n+1} \mathbf{D}^{(n+1)x_2}_{n} + \mathbf{D}^{nx_3}_{n+1} \mathbf{D}^{(n+1)x_3}_{n}.
\end{split}
\end{equation}
The other two equations in \eqref{eq:coefficient_1}, namely $\mathbf{B}_n=\widetilde{\mathbf{B}}_n$ and $\mathbf{C}_n=\widetilde{\mathbf{C}}_n$, yield similar relations to \eqref{eq:coefficient_1_same}--\eqref{eq:dr1}, by replacing $\mathbf{G}^n_{1,:}$, $\mathbf{D}^{nx_1}_{n+1}$, $\mathbf{D}^{nx_1}_{n-1}$ successively by $\mathbf{G}^n_{2,:}$, $\mathbf{D}^{nx_2}_{n+1}$, $\mathbf{D}^{nx_2}_{n-1}$ and $\mathbf{G}^n_{3,:}$, $\mathbf{D}^{nx_3}_{n+1}$, $\mathbf{D}^{nx_3}_{n-1}$. Putting these three equations together, we can get the final equation which can decide the coefficient matrix $\mathbf{G}^n$, i.e.
\begin{equation}\label{eq:coefficient_final_1}
  \left[
    \begin{array}{ccc}
      \mathbf{G}^n_{1,:} & \mathbf{G}^n_{2,:} & \mathbf{G}^n_{3,:} \\
    \end{array}
  \right] \mathbf{H} = \mathbf{0}.
\end{equation}
By direct calculation, the eigenvalues of the matrix $\mathbf{H}$ is
\begin{equation}
\begin{split}
  \zeta_1^{'} & = c^1_{n,\lambda,\mu} \left(c - \zeta_1 \right),  \\
  \zeta_2^{'} & = c^2_{n,\lambda,\mu} \left(c - \zeta_2 \right),\\
  \zeta_3^{'} & = c^3_{n,\lambda,\mu} \left(c - \zeta_3 \right),
\end{split}
\end{equation}
where
\begin{equation}
\begin{split}
  \zeta_1 & = -\frac{n+2}{n-1}, \\
  \zeta_2 & = -\frac{(2n+2)((n-1) \lambda + (2n-2) \mu)}{(2n^2 + 1) \lambda + (2 + 2n(n-1))\mu},\\
  \zeta_3 & = -\frac{(2n^2 + 4n + 3)\lambda + (2n^2 + 6n +6)\mu}{2n((n+2)\lambda + (3n + 5)\mu)},
\end{split}
\end{equation}
and $ c^1_{n,\lambda,\mu}$, $ c^1_{n,\lambda,\mu}$ and $ c^1_{n,\lambda,\mu}$ are non-zero constants which depend on $n$, $\lambda$ and $\mu$. Moreover, the multiplicities of the eigenvalues $\zeta_1^{'}$, $\zeta_2^{'}$ and $\zeta_3^{'}$ are, respectively, $2n+1$, $2n-1$, and $2n+3$. Furthermore, the numbers of linearly independent eigenfunctions of the eigenvalues $\zeta_1^{'}$, $\zeta_2^{'}$ and $\zeta_3^{'}$ are, respectively, $2n+1$, $2n-1$, and $2n+3$ and these eigenfunctions are orthogonal.

By setting
\begin{equation}
  c = \zeta_1,
\end{equation}
and solving the equation (\ref{eq:coefficient_final_1}), one can obtain the corresponding $2n+1$ solutions, namely $\mathbf{G}^{n,\zeta_1,k}$, $k=1,2,\ldots,2n+1$. By direct calculations, these solutions satisfy
\begin{equation}
  \mathbf{t}^1_{n+1} = \mathbf{0}, \quad \mathbf{t}^3_{n-1}=\mathbf{0}.
\end{equation}
Therefore the corresponding perfect plasmon elastic waves are
\begin{equation}
\widehat{\bm{\psi}}_{n,k}(x)=\begin{cases}
  \mathbf{G}^{n,\zeta_1,k}  r^n \mathbf{Y}_n(\hat{x}),               & r \leq R, \\
  \mathbf{G}^{n,\zeta_1,k} \frac{R^{2n+1}}{ r^{n+1} } \mathbf{Y}_n(\hat{x}), & r>R.
\end{cases}
\end{equation}

Next, by setting
\begin{equation}
  c = \zeta_2,
\end{equation}
and solving the equation (\ref{eq:coefficient_final_1}), one can obtain the corresponding $2n-1$ solutions, namely $\mathbf{G}^{n,\zeta_2,k}$, $k=1,2,\ldots,2n-1$. By direct calculations, these solutions satisfy
\begin{equation}
  \mathbf{t}^1_{n+1} = \mathbf{0}, \quad \mathbf{t}^3_{n-1} \neq \mathbf{0}.
\end{equation}
Therefore the corresponding perfect plasmon elastic waves are
\begin{equation}
\widehat{\bm{\psi}}_{n,k}(x)=\begin{cases}
  & \mathbf{G}^{n,\zeta_2,k}  r^n \mathbf{Y}_n(\hat{x})\\
 & - M_n(r^2 - R^2) \left[
                                                 \begin{array}{c}
                                                   \mathbf{t}^3_{n-1} \mathbf{D}^{(n-2)x_1}_{n-1} \\
                                                   \mathbf{t}^3_{n-1} \mathbf{D}^{(n-2)x_2}_{n-1}\\
                                                   \mathbf{t}^3_{n-1} \mathbf{D}^{(n-2)x_3}_{n-1}\\
                                                 \end{array}
                                               \right] r^{n-2} \mathbf{Y}_{n-2}(\hat{x})
 ,             \qquad  r \leq R, \medskip \\
 & \mathbf{G}^{n,\zeta_2,k} \frac{R^{2n+1}}{ r^{n+1} } \mathbf{Y}_n(\hat{x}),\hspace*{5.8cm}  r>R.
\end{cases}
\end{equation}

Finally, by setting
\begin{equation}
  c = \zeta_3,
\end{equation}
and solving the equation (\ref{eq:coefficient_final_1}), one can obtain the corresponding $2n+3$ solutions, namely $\mathbf{G}^{n,\zeta_3,k}$, $k=1,2,\ldots,2n+3$. By direct calculations, these solutions satisfy
\begin{equation}
  \mathbf{t}^1_{n+1} \neq \mathbf{0}, \quad \mathbf{t}^3_{n-1} = \mathbf{0}.
\end{equation}
Therefore the corresponding perfect plasmon elastic waves are
\begin{equation}
\widehat{\bm{\psi}}_{n,k}(x)=\begin{cases}
  & \mathbf{G}^{n,\zeta_3,k}  r^n \mathbf{Y}_n(\hat{x})  ,    \hspace*{5.8cm}          r \leq R,\medskip\\
  & \mathbf{G}^{n,\zeta_3,k} \frac{R^{2n+1}}{ r^{n+1} } \mathbf{Y}_n(\hat{x})\\
   &+ k_n(r^2 - R^2) \left[
                                                 \begin{array}{c}
                                                   \mathbf{t}^1_{n+1} \mathbf{D}^{(n+2)x_1}_{n+1} \\
                                                   \mathbf{t}^1_{n+1} \mathbf{D}^{(n+2)x_2}_{n+1}\\
                                                   \mathbf{t}^1_{n+1} \mathbf{D}^{(n+2)x_3}_{n+1}\\
                                                 \end{array}
                                               \right] r^{-n-3} \mathbf{Y}_{n+2}(\hat{x}) ,\ \ r>R.
\end{cases}
\end{equation}

The proof is complete.
\end{proof}

\begin{rem}\label{rem:ppw1}
For the coefficient matrices given in \eqref{eq:solution1}-\eqref{eq:solution12}, \eqref{eq:solution2}-\eqref{eq:solution22} and \eqref{eq:solution3}-\eqref{eq:solution32}, one can verify that there holds the following orthogonality relation,
\begin{equation}
  \mathbf{G}^{n,\zeta_i,k_1}:\overline{\mathbf{G}^{n,\zeta_j,k_2}}=0, \quad \text{if} \quad i \neq j, \quad \text{or}\quad i=j, k_1\neq k_2,
\end{equation}
where we recall that the operator $:$ between two matrices is defined in \eqref{eq:matrixnorm}. Thus from last equation, one has that
\begin{equation}
   \int_{\partial B_1} \mathbf{G}^{n,\zeta_i,k_1}\mathbf{Y}_n(\hat{x}) \cdot \overline{\mathbf{G}^{n,\zeta_j,k_2}\mathbf{Y}_n(\hat{x})}=0,
\end{equation}
 if $i\neq j$, or $i=j$, $k_1 \neq k_2$. Throughout our subsequent study, we normalize $\mathbf{G}^{n,\zeta_i,k}$, $i=1,2,3$, such that
\begin{equation}
  \mathbf{G}^{n,\zeta_i,k}:\overline{\mathbf{G}^{n,\zeta_i,k}}=1,
\end{equation}
and then one can show that
\begin{equation}
  \int_{\partial B_1} \mathbf{G}^{n,\zeta_i,k}\mathbf{Y}_n(\hat{x}) \cdot \overline{\mathbf{G}^{n,\zeta_i,k}\mathbf{Y}_n(\hat{x})}=\mathbf{G}^{n,\zeta_i,k}:\overline{\mathbf{G}^{n,\zeta_i,k}}=1.   
\end{equation}

\end{rem}

\begin{rem}\label{rem:NP}

By the strong convexity condition \eqref{eq:convex}, one can readily verify that $\zeta_i$, $i=1,2,3$ defined in \eqref{eq:coefficient_pi} are all negatively valued. The choice of those plasmon constants are crucial for the existence of the perfect plasmon waves. In \cite{AKKY,LiLiu2d}, the plasmon constant and the corresponding perfect plasmon waves are determined by use of the spectral properties of the associated Neumann-Poincar\'e operator as follows. Let the Kelvin matrix of the fundamental solution  $\bm{\Phi}=(\Phi_{ij})_{i,j=1}^N$ to the PDO $\mathcal{L}_{\lambda,\mu}$ is given by (cf. \cite{Kup})
\begin{equation}\label{eq:funds}
\Phi_{ij}(x)=\begin{cases}
\displaystyle{\frac{\alpha}{2\pi}\delta_{ij}\ln\|x\|-\frac{\beta}{2\pi}\frac{x_ix_j}{\|x\|^2}}\quad &\mbox{when}\ \  N=2,\medskip\\
\displaystyle{-\frac{\alpha}{4\pi}\frac{\delta_{ij}}{\|x\|}-\frac{\beta}{4\pi}\frac{x_ix_j}{\|x\|^3}}\quad &\mbox{when}\ \ N=3,
\end{cases}
\end{equation}
where
\[
\alpha=\frac 1 2\left(\frac 1 \mu+\frac{1}{2\mu+\lambda}\right)\quad\mbox{and}\quad \beta=\frac 1 2\left(\frac 1 \mu-\frac{1}{2\mu+\lambda}\right),
\]
and $x=(x_i)_{i=1}^N\in\mathbb{R}^N$ and $\delta_{ij}$ is the Kronecker delta. We seek a solution to \eqref{eq:trial1} using the integral ansatz as a single-layer potential as follows 
\begin{equation}\label{eq:asz1}
\bm{\psi}=\mathbf{S}[\bm{\varphi}](x) := \int_{\partial B_R} \bm{\Phi}(x-y)\bm{\varphi}(y)\, ds(y) \quad x\in \mathbb{R}^2,
\end{equation}
where $\bm{\varphi}\in H^{-1/2}(\partial B_R)^N$. There holds the following jump relationship of the conormal derivative of the single layer potential in \eqref{eq:asz1} 
\begin{equation}\label{eq:jump}
  \frac{\partial}{\partial \bm{\nu}_{\lambda_A,\mu_A}} \mathbf{S}[\bm{\varphi}]\big|_{\pm} = \left( \pm \frac{1}{2} I + \mathbf{K}^*  \right)[\bm{\varphi}] \quad \text{on} \quad \partial B_R,
\end{equation}
where
\begin{equation}
  \mathbf{K}^* [\bm{\varphi}] =\mathrm{p.v.} \int_{\partial B_R} \frac{\partial}{\partial \bm{\nu}_{\lambda_A,\mu_A}(x)} \bm{\Phi}(x-y) \bm{\varphi}(y) d s(y) \quad x \in \partial B_R,
\end{equation}
and $\mathrm{p. v.}$ stands for the Cauchy principle value. Using the transmission condition across $\partial B_R$ for $\bm{\psi}$ (see \eqref{eq:trial1}), along with the help of \eqref{eq:jump}, one can show that 
\begin{equation}\label{eq:NPn1}
  \mathbf{K}^*[\bm{\varphi}] = \frac{c+1}{2(c-1)} \bm{\varphi}.
\end{equation}
Hence, in order to find a nontrivial solution $\bm{\psi}$ in \eqref{eq:asz1}, $\frac{c+1}{2(c-1)}$ should belong to the spectral set of the Neumann-Poincar\'e operator $\mathbf{K}^*$. The spectral properties of $\mathbf{K}^*$ in the two-dimensional case is now known by the study in \cite{AKKY}, but the three-dimensional case is not yet understood. Clearly, by Theorem~\ref{thm:perfectwaves}, one can readily derive the eigenvalues and the corresponding eigenfunctions for this Neumann-Poincar\'e operator $\mathbf{K}^*$ in the 3D spherical geometry. 
\end{rem}

In the next section, we shall apply the nontrivial solutions $\widehat{\bm{\psi}}_{n, k}$ in Theorem~\ref{thm:perfectwaves} as trial functions to the primal and dual variational principles in Theorem~\ref{thm:primaldual} to establish the resonance and non-resonance results. In the variational principles, the trial functions are all real-valued, but the solutions $\widehat{\bm{\psi}}_{n, k}$ found in Theorem~\ref{thm:perfectwaves} are all complex-valued. Nevertheless, we would like to point out that clearly the real-valued functions $\Re\widehat{\bm{\psi}}_{n, k}$ and $\Im\widehat{\bm{\psi}}_{n, k}$ are all solutions to \eqref{eq:trial1}.

\section{Plasmon resonances for the elastostatic system}\label{sect:4}

With the preparations in Sections 2 and 3, we are in a position to consider the anomalous localised resonance for a plasmonic device of the form \eqref{eq:tensor1}-\eqref{eq:tensor2}, associated with the elastostatic system \eqref{eq:lame1}.
 {\color{black}
 Henceforth, we assume that the force term $\bff(x)$ is a real-valued distributional functional of the following form
 }
\begin{equation}\label{eq:force1}
\mathbf{f}=\mathbf{F}\mathcal{H}^2 \lfloor \partial B_q,\quad \mathbf{F}: \partial B_q\rightarrow \mathbb{R}^3,\quad \mathbf{F}\in L^2(\partial B_q)^3,\ \ q\in\mathbb{R}_+,
\end{equation}
and
\begin{equation}\label{eq:force2}
\int_{\partial B_q} \mathbf{F}\ d \mathcal{H}^2=0.
\end{equation}
Next we give the Fourier series expression of the force term $\bff$ specified in \eqref{eq:force1} and \eqref{eq:force2}, which can be represented as follows
\begin{equation}\label{eq:source2}
  \bff = \sum_{n=1}^{\infty}\left( \sum_{k=1}^{2n+1} \gamma_{n,\zeta_1,k}  \mathbf{G}^{n,\zeta_1,k} + \sum_{k=1}^{2n-1} \gamma_{n,\zeta_2,k} \mathbf{G}^{n,\zeta_2,k} + \sum_{k=1}^{2n+3} \gamma_{n,\zeta_3,k} \mathbf{G}^{n,\zeta_3,k}  \right) \mathbf{f}_n^q,
\end{equation}
where
\begin{equation}\label{eq:fc1}
\begin{split}
  \mathbf{f}_n^q=& \mathbf{Y}_n \mathcal{H}^2 \lfloor \partial B_q,\\
  \gamma_{n,\zeta_i,k}=& \int_{\partial B_1} \bff(q\hat{x})\cdot \overline{\mathbf{G}^{n,\zeta_i,k}\mathbf{Y}_n(\hat{x})} ds(\hat{x}),
  \end{split}
\end{equation}
and $\mathbf{G}^{n,\zeta_i,k}$, $i=1,2,3$ are given in Theorem \ref{thm:perfectwaves}.

Throughout the rest of the paper, we assume that $\mathbf{G}^{1,\zeta_i,k}=\mathbf{0}$, $i=1,2,3$. Finally, we let the exterior domain $\Omega$ for the plasmonic structure \eqref{eq:tensor2} be taken to be $B_R$ with a fixed $R\in\mathbb{R}_+$ and $R<q$.

\subsection{Non-resonance result}

We first show that the plasmonic device considered in \cite{LiLiu2d} that is resonant in $\mathbb{R}^2$ does not induce resonance in $\mathbb{R}^3$, which necessitates our subsequent design of novel plasmonic devices in Sections 4.2 and 4.3 for the occurrence of resonances.

\begin{thm}\label{thm:nonresoance}
 Consider the elastic configuration $(\mathbf{C}_{\widetilde\lambda,\widetilde\mu},\bff)$, where $\mathbf{C}_{\widetilde\lambda,\widetilde\mu}$ is described in  \eqref{eq:tensor1}--\eqref{eq:tensor2} with $\Sigma=B_1$, $\Omega=B_{r_e}$ and $c$ being a negative-valued constant. Let $\bff$ be given in (\ref{eq:source2}) with $\gamma_{n,\zeta_i,k}=0$, $i=2,3$. Then there is no resonance.
\end{thm}

\begin{proof}
  We make use primal variation principle in Theorem~\ref{thm:primaldual} to prove the non-resonance by finding suitable test functions $(\mathbf{v}_\delta,\mathbf{w}_\delta)$ satisfying $ \mathcal{L}_{\lambda_A,\mu_A}\mathbf{v}_{\delta}-\mathcal{L}_{\lambda,\mu}\mathbf{w}_{\delta}=\mathbf{f}$ in (\ref{eq:primal}) such that $\mathbf{I}_\delta(\mathbf{v}_\delta,\mathbf{w}_\delta)$ remains bounded as $\delta \rightarrow +0$.

 For $n \in \mathbb{N}$, we set
\begin{equation}\label{eq:def3}
  \widehat{\mathbf{v}}_{n,k}=\left\{
                 \begin{array}{ll}
                  \mathbf{G}^{n,\zeta_1,k} r^n \mathbf{Y}_n , & \quad \quad  |x| \leq 1, \\
                  e_1 \mathbf{G}^{n,\zeta_1,k} r^n \mathbf{Y}_n + e_2 \mathbf{G}^{n,\zeta_1,k} \frac{\mathbf{Y}_n}{r^{n+1}}, & \quad \quad  1< |x| \leq r_e, \\
                  e_3 \mathbf{G}^{n,\zeta_1,k} r^n \mathbf{Y}_n + e_4 \mathbf{G}^{n,\zeta_1,k} \frac{\mathbf{Y}_n}{r^{n+1}}, & \quad \quad  r_e< |x| \leq q, \\
                  e_5 \mathbf{G}^{n,\zeta_1,k} \frac{\mathbf{Y}_n}{r^{n+1}}, & \quad \quad  q< |x|,
                 \end{array}
               \right.
\end{equation}
where
\begin{align*}
    e_1 =& \frac{n-1 + c (n+2)} {c (2n +1)} ,\\
  e_2 =& \frac{(c-1)(n-1)}{c (2n +1)},\\
  e_3 =& \frac{-(c-1)^2(n^2+n-2) + (2+ c(n-1)+n)(n-1+c(n+2) )r_e^{2n+1}}{c (1+2n)^2 r_e^{2n+1}},\\
  e_4 =& \frac{-(c-1)(n-1)( c(n+2) +n-1 ) (r_e^{2n+1}-1)}{c(2n+1)^2 } ,\\
e_5 =&   \frac{-(c-1)(n-1)(n-1+c(n+2))(r_e^{2n+1}-1)  } {c (2n +1)^2} + \\
      & \frac{   q^{2n+1} \left( -(c-1)^2(n^2+n-2)r_e^{-2n-1} +(2+c(n-1)+n )(-1+n+ c(n+2)  )  \right) } {c (2n +1)^2}.
 \end{align*}
 By direct calculations, one can show that $ \widehat{\mathbf{v}}_{n,k}$ defined above is continuous over $\mathbb{R}^3$ and satisfies
\begin{equation}
   \mathcal{L}_{\lambda_A,\mu_A}\widehat{\mathbf{v}}_{n,k}=\mathbf{0} \quad x\in \mathbb{R}^3\backslash\partial B_q.
\end{equation}
On $\partial B_q$, $\widehat{\mathbf{v}}_{n,k}$ has a conormal derivative:
\begin{equation}
  \frac{\partial \widehat{\bv}_{n,k}}{\partial \bm{\nu}_{\lambda,\mu}}=e_6 \mathbf{G}^{n,\zeta_1,k} \mathbf{Y}_n,
\end{equation}
where $\frac{\partial \widehat{\bv}_{n,k}}{\partial \bm{\nu}_{\lambda,\mu}}$ is defined in (\ref{eq:normald}) and
\begin{equation}
  e_6=\frac{(c-1)^2(n^2+n-2) -(n+2+ c(n-1) )( n-1 + c(n+2)  )r_e^{2n+1} }{c(2n+1)q} \frac{q^n}{r_e^{2n+1}}.
\end{equation}
Therefore, by setting
\begin{equation}\label{coeff_tau}
  \tau_{n,k}=\gamma_{n,k}/e_6
\end{equation}
one can readily verify that
\begin{equation}
   \mathcal{L}_{\lambda_A,\mu_A}\left( \tau_{n,k} \widehat{\mathbf{v}}_{n,k} \right)=\gamma_{n,k} \mathbf{G}^{n,\zeta_1,k}  \mathbf{Y}_n.
\end{equation}
Hence, by letting
\begin{equation}
  \bv_{\delta}=\sum_{n,k} \tau_{n,k} \widehat{\bv}_{n,k},
\end{equation}
we shall have
\begin{equation}
  \mathcal{L}_{\lambda_A,\mu_A} \bv_{\delta} = \bff \quad \text{in}\ \ \mathbb{R}^3.
\end{equation}
Therefore, if we set $\bw_{\delta} \equiv \mathbf{0}$, then $(\bv_\delta,\bw_\delta)$ satisfies the PDE constraint in (\ref{eq:primal}).

Finally, by the primal variational principle in Theorem~\ref{thm:primaldual}, we can deduce the following estimate:
\begin{equation}
\begin{split}
  \mathbf{E}_\delta(\mathbf{u}_\delta)\leq & \mathbf{I}_\delta(\mathbf{v}_\delta,\mathbf{w}_\delta)  = \mathbf{I}_\delta(\mathbf{v}_\delta,0)=\frac\delta 2 \mathbf{P}_{\lambda,\mu}(\mathbf{v}_{\delta},\mathbf{v}_{\delta})\\
    & \leq \delta C \sum_{n,k} \tau_n^2(\mathbf{G}^{n,\zeta_1,k}:\overline{\mathbf{G}^{n,\zeta_1,k}} ) n q^{2n} \leq \delta C  \sum_n \frac{n q^{2n}}{n^2 q^{2n}} \|\mathbf{F}\|^2_{L^2(\partial B_q)}\\
    &  \leq \delta C \|\mathbf{F}\|^2_{L^2(\partial B_q)}.
\end{split}
\end{equation}
That is, there is no resonance and the proof is complete.
\end{proof}

By Theorem~\ref{thm:nonresoance}, if the negative plasmon constant $c$ is fixed, then there is no resonance occurring. In \cite{AKKY,LiLiu2d}, two-dimensional plasmon resonances are shown to occur for plasmonic devices with a fixed negative plasmon constant being $c=-1$. The major difference between the studies in \cite{AKKY} and \cite{LiLiu2d} is that the loss parameter $\delta$ only exists in the plasmonic layer in \cite{AKKY}, whereas the loss parameter exists in the whole space in \cite{LiLiu2d}. In Theorem~\ref{thm:nonresoance}, we have assumed that the loss parameter exists in the whole space. Nevertheless, it might be unobjectionable to claim that resonance does not occur even if the loss parameter only exists in the plasmonic layer. In Theorem~\ref{thm:nocore} in what follows, we shall show that if the core $\Sigma$ in Theorem~\ref{thm:nonresoance} is taken to be an empty set, one can properly choose a negative-valued plasmon constant, depending on the force term $\bff$, such that resonance occurs.

\subsection{Resonance with no core}

In this section, we shall construct a novel plasmonic device without a core, namely $\Sigma=\emptyset$, which ensures that the resonance can always occur.

\begin{thm}\label{thm:nocore}
Let $\bff$ be given by (\ref{eq:source2}), with $\gamma_{n_0,\zeta_i,k} \neq 0$ for some $n_0\in\mathbb{N}$ and $1\leq i\leq 3$, representing the force supported at a distance $q>R$. Consider the elastic configuration $(\mathbf{C}_{\widetilde\lambda,\widetilde\mu},\bff)$, where $\mathbf{C}_{\widetilde\lambda,\widetilde\mu}$ is described in \eqref{eq:tensor1}--\eqref{eq:tensor2} with $c=\zeta_i$, given in \eqref{eq:coefficient_pi}, $n=n_0$ and $\Omega=B_R$.  Assume that there is no core; that is, $\Sigma=\emptyset$. Then the configuration is resonant, i.e. $\mathbf{E}_{\delta}(\mathbf{C}_{\widetilde\lambda,\widetilde\mu},\bff) \rightarrow +\infty$ as $\delta \rightarrow +0$.
\end{thm}

\begin{proof}
In what follows, we only give the proof of the situation when $\gamma_{n_0,\zeta_1,k}\neq 0$, and when $\gamma_{n_0,\zeta_2,k}\neq 0$ or $\gamma_{n_0,\zeta_3,k}\neq 0$, the discussion is similar.

We shall make use of the dual variational principle for its proof. Fix the radii $R, q$ and consider an arbitrary sequence $\delta=\delta_j \rightarrow +0$ as $j\rightarrow+\infty$. Our aim is to find a sequence $(\bv_{\delta},\bm{\psi}_{\delta})$, satisfying the constraint $\cl_{\lambda_A,\mu_A} \bm{\psi}_{\delta} + \delta \cl_{\lambda,\mu}\bv_{\delta}=\mathbf{0}$ of \eqref{eq:dual} and such that $\mathbf{J}_{\delta}(\bv_{\delta},\bm{\psi}_{\delta})\rightarrow +\infty $.

Since $\gamma_{n_0,\zeta_1,k} \neq 0$, we first assume that $\Re{\gamma_{n_0,\zeta_1,k}} \neq 0$ and choose
\begin{equation}
  \bv_{\delta} \equiv \mathbf{0}
\end{equation}
\begin{equation}
  \bm{\psi}_{\delta} : \equiv \tau_{\delta} \Re\overline{\widehat{\bm{\psi}}_{n_0,k}},
\end{equation}
where $\widehat{\bm{\psi}}_{n_0,k}$ is given by (\ref{eq:solution1}) and $\tau_{\delta} \in \mathbb{R}$ satisfies $\Re\gamma_{n_0,\zeta_1,k}\cdot\tau_{\delta}>0$ and will be further chosen below. Thus the pair $(\bv_{\delta},\bm{\psi}_{\delta})$ satisfies the PDE constraint in (\ref{eq:dual}). With the help of dual variational principle in Theorem~\ref{thm:primaldual}, the definition of $\mathbf{J}_{\delta}$, the orthogonality of Fourier series and $\gamma_{n_0,\zeta_1,k} \neq 0$ for some $n_0\in\mathbb{N}$, we can obtain
\begin{equation}
  \begin{split}
    \mathbf{E}_{\delta}(\bu_{\delta}) & \geq  \mathbf{J}_{\delta}(\bv_{\delta},\bm{\psi}_{\delta}) = \mathbf{J}_{\delta}( \mathbf{0} ,\bm{\psi}_{\delta}) =\int \bff \cdot \bm{\psi}_{\delta} -\frac{\delta}{2}\mathbf{P}_{\lambda,\mu}(\bm{\psi}_{\delta},\bm{\psi}_{\delta}) \\
      & =\Re \int_{\partial B_q} \gamma_{n_0,\zeta_1,k} \tau_{\delta} q^{-n_0} R^{2n_0}\left( \mathbf{G}^{n_0,\zeta_1,k}\mathbf{Y}_{n_0} \cdot \overline{\mathbf{G}^{n_0,\zeta_1,k}\mathbf{Y}_{n_0}}\right) -\frac{\delta}{2} |\tau_{\delta}|^2 \mathbf{P}_{\lambda,\mu}(\widehat{\bm{\psi}}_{n_0},\widehat{\bm{\psi}}_{n_0}) \\
      & = C_0 \tau_{\delta} - C_1(\delta |\tau_{\delta}|^2 ) ,
  \end{split}
\end{equation}
Here and in what follows, $C_0$ and $C_1$ are two generic constants depending on $\mathbf{G}^{n_0,\zeta_1,k}$, $\gamma_{n_0,\zeta_1,k}$, $q$ and $R$ and may change from one inequality/equality to another.

Choosing $\tau_{\delta} \rightarrow \infty$ with $\delta |\tau_{\delta}|^2 \rightarrow +0$ as $\delta\rightarrow+0$, we obtain $ \mathbf{E}_{\delta}(\bu_{\delta})\rightarrow +\infty  $ for $\delta \rightarrow +0$.

Next, if $\Im \gamma_{n_0,\zeta_1,k} \neq 0$, by setting
\begin{equation}
   \bv_{\delta} = \mathbf{0} \quad   \mbox{and} \quad   \bm{\psi}_{\delta} = \tau_{\delta} \Im\overline{\widehat{\bm{\psi}}_{n_0,k}},
\end{equation}
and using a similar argument as before, one can show that the resonance occurs.

The proof is complete.
\end{proof}

\subsection{Resonance with a core of an arbitrary shape}\label{sect:ALR1}

In this section, we consider a non-radial geometry with a nonempty core, $\Sigma \subset B_1$, of an arbitrary shape. We shall show that the resonance or non-resonance of the configuration strongly depends on the location of the force term $\bff$.

\begin{thm}\label{thm:main2}
Consider the elastic configuration $(\mathbf{C}_{\widetilde\lambda,\widetilde\mu},\bff)$, where $\mathbf{C}_{\widetilde\lambda,\widetilde\mu}$ is described in \eqref{eq:tensor1}--\eqref{eq:tensor2} with $\Omega=B_R$ for a certain $R>1$ and $\Sigma\subset B_1$ with a connected Lipschitz boundary $\partial \Sigma$. Furthermore, let $n_\delta$ be the smallest integer such that
\begin{equation}\label{eq:aaa1}
 R^{-n_{\delta}} < \delta,
 \end{equation}
and $c=\zeta_i$, $1\leq i\leq 3$, where $\zeta_i$ is given in \eqref{eq:coefficient_pi} with $n=n_\delta$. Consider the elastostatic system (\ref{eq:lame1}), with $\mathbf{C}_{\widetilde\lambda,\widetilde\mu}$ described above.  Then for the source $\bff$ of the form \eqref{eq:source2} only with $\gamma_{n,\zeta_i,k} \neq 0$, namely $\gamma_{n,\zeta_j,k}=0$, $j\neq i$, supported at a distance $q$ from the origin with $R < q < R^*:=R^{3/2}$, the configuration $(\mathbf{C}_{\widetilde\lambda,\widetilde\mu},\bff)$ is resonant.
\end{thm}

\begin{proof}

In what follows, we only give the proof when $c=\zeta_1$ with
\begin{equation}\label{eq:source_pi_1}
 \bff = \sum_{n=1}^{\infty}\left( \sum_{k=1}^{2n+1}  \gamma_{n,\zeta_1,k} \mathbf{G}^{n,\zeta_1,k}  \right) \mathbf{f}_n^q,
\end{equation}
and when $c=\zeta_2$ or $c=\zeta_3$, the corresponding discussion is similar

 We fix $R < q < R^*$ and a sequence $\delta=\delta_j \rightarrow +0$ and consider a force term $\bff$ given by (\ref{eq:source_pi_1}). Our aim is to find a sequence $(\bv_{\delta},\bm{\psi}_{\delta})$, satisfying the PDE constraint $\mathcal{L}_{\lambda_A,\mu_A} \bm{\psi}_{\delta} + \delta \cl_{\lambda,\mu}\bv_{\delta}=\mathbf{0}$ in (\ref{eq:dual}) and such that $\mathbf{J}_{\delta}(\bv_{\delta},\bm{\psi}_{\delta})\rightarrow +\infty $.

First, we choose
\begin{equation}\label{eq:test_function}
  \bm{\psi}_{\delta} : \equiv \tau_{\delta} \Re \overline{\widehat{\bm{\psi}}_{n_{\delta},k}},
\end{equation}
where $\widehat{\bm{\psi}}_{n_{\delta},k}$ is given by (\ref{eq:solution1}). The number $\tau_{\delta} \in \mathbb{R}$ will be properly chosen below. For $\bm{\psi}_{\delta}$, it is apparent that $\mathcal{L}_{\lambda_A,\mu_A} \bm{\psi}_{\delta} \neq\mathbf{0}$ along the core interface $\partial \Sigma \subset B_1$. In order to satisfy the PDE constraint we define $\bv_{\delta}$ to be the solution of $ -\delta \cl_{\lambda,\mu}\bv_{\delta} = \mathcal{L}_{\lambda_A,\mu_A} \bm{\psi}_{\delta}$. Since $- \cl_{\lambda,\mu}$ is an elliptic PDO, by the standard elliptic estimates one can arrive at the following estimate:
\begin{equation}
  \delta \mathbf{P}_{\lambda,\mu}(\bv_{\delta}, \bv_{\delta}) \leq C_0 \delta^{-1} \| \mathcal{L}_{\lambda_A,\mu_A} \bm{\psi}_{\delta} \|^2_{H^{-1}(\mathbb{R}^3)^3} \leq C_0 \delta^{-1} \tau_{\delta}^2 n_{\delta}.
\end{equation}

It remains to calculate the energy $\mathbf{J}_{\delta}(\bv_{\delta},\bm{\psi}_{\delta})$. Since $n_{\delta}$ is the smallest integer fulfilling \eqref{eq:aaa1}, one clearly has that $R^{-n_{\delta} +1} \geq \delta$. With the help of the dual variational principle in Theorem~\ref{thm:primaldual}, we have
\begin{equation}
  \begin{split}
    \mathbf{E}_{\delta}(\bu_{\delta}) & \geq \mathbf{J}_{\delta}(\bv_{\delta},\bm{\psi}_{\delta})= \int \bff \cdot \bm{\psi}_{\delta} - \frac{\delta}{2} \mathbf{P}_{\lambda,\mu}(\bv_{\delta},\bv_{\delta}) - \frac{\delta}{2}\mathbf{P}_{\lambda,\mu}(\bm{\psi}_{\delta},\bm{\psi}_{\delta})   \\
      & \geq (C_0 \Re\gamma_{n_{\delta},\zeta_1,k} \tau_{\delta} q^{-n_{\delta}} R^{2n_{\delta}} -C_1 \delta^{-1} \tau_{\delta}^2 n_{\delta} - C_1 \delta  \tau_{\delta}^2 n_{\delta} R^{2 n_{\delta}} )(\mathbf{G}^{n_{\delta},\zeta_1,k}:\overline{\mathbf{G}^{n_{\delta},\zeta_1,k}})\\
      & \geq \tau_{\delta} R^{n_{\delta}} (\mathbf{G}^{n_{\delta},\zeta_1,k}:\overline{\mathbf{G}^{n_{\delta},\zeta_1,k}}) \bigg(  C_0 \Re\gamma_{n_{\delta},\zeta_1,k} \left(\frac{R}{q} \right)^{n_{\delta}}\\
      &\hspace*{4cm} - C_1 \frac{1}{(\delta R^{n_{\delta}} )} \tau_{\delta} n_{\delta} - C_1 \tau_{\delta} n_{\delta}(\delta R^{n_{\delta}})   \bigg).
  \end{split}
\end{equation}
The choice of $R^{n_{\delta}} < \delta$ with $1< \delta R^{n_{\delta}} \leq R$ ensures that the last two contributions are of comparable order. We then find, for some $C_1>0$,
\begin{equation}\label{eq:dd1}
  \mathbf{E}_{\delta}(\bu_{\delta}) \geq \tau_{\delta} R^{n_{\delta}}  \left(  C_0 \Re\gamma_{n_{\delta},\zeta_1,k} \left(\frac{R}{q} \right)^{n_{\delta}}  - C_1 \tau_{\delta} n_{\delta}  \right).
\end{equation}
We choose $\tau_{\delta}$ to be
\begin{equation}\label{eq:dd2}
  \tau_{\delta}= \frac{1}{2 C_1 n_{\delta}  } C_0 \Re\gamma_{n_{\delta},\zeta_1,k} \left(\frac{R}{q} \right)^{n_{\delta}},
\end{equation}
and then from \eqref{eq:dd1} and \eqref{eq:dd2} we readily have that
\begin{equation}\label{eq:dd3}
  \mathbf{E}_{\delta}(\bu_{\delta}) \geq \tau_{\delta} R^{n_{\delta}} \left( \frac{1}{2} C_0 \Re\gamma_{n_{\delta},\zeta_1,k} \left(\frac{R}{q} \right)^{n_{\delta}}  \right) = \frac{1}{4 C_1 n_{\delta} } (C_0 \Re\gamma_{n_{\delta},\zeta_1,k})^2 \left(\frac{R^3}{q^2} \right)^{n_{\delta}}.
\end{equation}
By the assumption, $q<R^*$ and if the sequence of the Fourier coefficients $\gamma_{n,,\zeta_1,k}$ of the force term $\mathbf{f}$ decays not very quickly (ensuring that the RHS term of \eqref{eq:dd3} goes to infinity as $\delta\rightarrow+0$), we easily see from \eqref{eq:dd3} that $\mathbf{E}_{\delta}(\bu_{\delta}) \rightarrow +\infty$ as $\delta \rightarrow +0$.

Finally, by setting
\begin{equation}
  \bm{\psi}_{\delta} : \equiv \tau_{\delta} \Im \overline{\widehat{\bm{\psi}}_{n_{\delta},k}},
\end{equation}
in (\ref{eq:test_function}), by a similar argument as above, one can have a similar conclusion.

Thus, this proof is complete.
\end{proof}

\subsection{Non-resonance in the radial case}

In Section~\ref{sect:ALR1}, we show that for certain force terms lying within the critical radius $R^*$, the resonance occurs. In this section, we shall show that if the force terms lying outside the critical radius, then resonance does not occur. To that end, we would consider our study in the radial geometry by assuming that the core $\Sigma=B_1$.

\begin{thm}\label{thm:main3}
Consider the elastic configuration $(\mathbf{C}_{\widetilde\lambda,\widetilde\mu},\bff)$, where $\mathbf{C}_{\widetilde\lambda,\widetilde\mu}$ is described in \eqref{eq:tensor1}--\eqref{eq:tensor2} with $c=\zeta_1$ and, $\Omega=B_R$ for a certain $R>1$ and $\Sigma= B_1$. Consider the elastostatic system (\ref{eq:lame1}), with $\mathbf{C}_{\widetilde\lambda,\widetilde\mu}$ describe above. Let the source $\bff$ be given by (\ref{eq:source2}) with $\gamma_{n,\zeta_i,k}=0$, $i=2,3$, supported at the distance $q$ with $q>R^* := R^{3/2}$, then the configuration $(\mathbf{C}_{\widetilde\lambda,\widetilde\mu},\bff)$ is non-resonant.
\end{thm}

\begin{proof}
We make use of the primal variational principle to show the non-resonance result. We shall construct the test function $(\bv_{\delta},\bw_{\delta})$, satisfying the constraint
\begin{equation}\label{eq:thm_constraint1}
   \mathcal{L}_{\lambda_A,\mu_A} \bv_{\delta} -\cl_{\lambda,\mu} \bw_{\delta} =\bff
\end{equation}
such that the energy along this sequence , $\mathbf{I}_{\delta}(\bv_{\delta},\bw_{\delta})$ remains bounded.

Let $\bv_{\delta}$ be of the form:
\begin{equation}
  \bv_{\delta} = \sum_{n\neq n_\delta} \sum_{k} ( \bv_{\delta, n, k} + \bv_{\delta, n_\delta, k} ),
\end{equation}
where $\bv_{\delta, n, k}$, $n \neq n_{\delta}$, satisfies
\begin{equation}\label{eq:relation_1}
   \mathcal{L}_{\lambda_A,\mu_A} \bv_{\delta,n,k}= \gamma_{n,\zeta_1,k} \mathbf{G}^{n,\zeta_1,k} \mathbf{f}_n^q \quad \mbox{in} \quad \mathbb{R}^3,
\end{equation}
and $\bv_{\delta, n_{\delta}, k}$ satisfies
\begin{equation}\label{eq:relation_2}
   \mathcal{L}_{\lambda_A,\mu_A} \bv_{\delta,n_{\delta},k}= \gamma_{n_{\delta},\zeta_1,k} \mathbf{G}^{n_{\delta},\zeta_1,k} f_{n_{\delta}}^q \quad \mbox{on} \quad \partial B_q.
\end{equation}
By setting
\begin{equation}
  \bv_{\delta,n,k} = \tau_{n,k} \widehat{\bv}_{n,k},
\end{equation}
where $\tau_{n,k}$ and $\widehat{\bv}_{n,k}$ are given in (\ref{coeff_tau}) and (\ref{eq:def3}) with $c$ replaced by $-1-\frac{3}{n_{\delta}-1}$, respectively, one can readily verify that $\bv_{\delta, n, k}$, $n \neq n_{\delta}$, defined above satisfies (\ref{eq:relation_1}). Next we define
\begin{equation}
  \widehat{\bV}_{n_\delta,k}=\left\{
                             \begin{array}{ll}
                               \mathbf{G}^{n_\delta,\zeta_1,k} r^{n_\delta} \mathbf{Y}_{n_\delta}, & \quad \|x\| \leq q,\medskip \\
                               \mathbf{G}^{n_\delta,\zeta_1,k} q^{2 n_\delta +1} r^{-n_\delta -1} \mathbf{Y}_{n_\delta}, & \quad \|x\| > q,
                             \end{array}
                           \right.
\end{equation}
and set
\begin{equation}\label{eq:tau}
  \bv_{\delta, n_\delta, k}=\tau_{n_\delta,k} \widehat{\bV}_{n_\delta,k}, \qquad \tau_{n_\delta,k} = \frac{-\gamma_{n_\delta,\zeta_1,k}}{(2n_\delta +1) q^{n_\delta-1}},
\end{equation}
then it is easy to see that (\ref{eq:relation_2}) holds.

In order to satisfy the constraint (\ref{eq:thm_constraint1}), we choose $\bw_{\delta}$ as follows:
\begin{equation}\label{eq:w}
\begin{split}
   -\cl_{\lambda,\mu} \bw_{\delta} & =\bff -\mathcal{L}_{\lambda_A,\mu_A} \bv_{\delta} \\
    & =\sum_k \left(\bff_{n_\delta,k} - \mathcal{L}_{\lambda_A,\mu_A} \bv_{\delta, n_\delta, k}\right),
\end{split}
\end{equation}
where
\begin{equation}
  \bff_{n_\delta,k} = \gamma_{n_{\delta},\zeta_1,k} \mathbf{G}^{n_\delta,\zeta_1,k} f_{n_{\delta}}^q,
\end{equation}
and the last equation follows from (\ref{eq:relation_1}) and (\ref{eq:relation_2}).

It remains to calculate the energy $\mathbf{I}_{\delta}(\bv_{\delta},\bw_{\delta})$. First, for $\bv_{\delta, n_\delta,k}$, one has the following estimate:
\begin{equation}
  \delta \mathbf{P}_{\lambda,\mu}(\bv_{\delta, n_\delta,k}, \bv_{\delta, n_\delta,k}) = C \delta |\tau_{n_\delta,k}|^2 n_\delta q^{2n_\delta} \leq C \delta |\gamma_{n_\delta,k}|^2,
\end{equation}
and from the proof of theorem \ref{thm:nonresoance}, it is easy to have for $\bv_{\delta, n,k}$, $n \neq n_{\delta}$ that
\begin{equation}
   \delta \mathbf{P}_{\lambda,\mu}(\bv_{\delta,n,k}, \bv_{\delta,n,k}) \leq C\delta |\gamma_{n,k}|^2.
\end{equation}
Finally, we have
\begin{equation}
  \delta \mathbf{P}_{\lambda,\mu}(\bv_{\delta}, \bv_{\delta}) \leq C\delta \|\mathbf{F}\|^2_{L^2(\partial B_q)}.
\end{equation}
Next we estimate the energy due to $\bw_\delta$, and by (\ref{eq:w}) one has
\begin{equation}
  \frac{1}{\delta} \mathbf{P}_{\lambda,\mu}(\bw_{\delta}, \bw_{\delta}) \leq C \frac{1}{\delta} \|\sum_k \left( \bff_{n_\delta,k} - \mathcal{L}_{\lambda_A,\mu_A} \bv_{\delta, n_\delta,k} \right) \|^2_{H^{-1}} \leq C \frac{1}{\delta} \sum_k |\tau_{n_\delta,k}|^2 n_{\delta} R^{2n_\delta}.
\end{equation}
With the help of (\ref{eq:tau}) and the choice of $n_\delta$, $R^{-n_\delta} \approx \delta$, we have
\begin{equation}
  \frac{1}{\delta} \mathbf{P}_{\lambda,\mu}(\bw_{\delta}, \bw_{\delta}) \leq C \frac{1}{\delta} \sum_k \frac{\gamma_{n_\delta,k}^2}{n_\delta} \left(\frac{R}{q} \right)^{2n_\delta} \leq C \sum_k \frac{\gamma_{n_\delta,k}^2}{n_\delta} \left(\frac{R^{3/2}}{q} \right)^{2n_\delta}.
\end{equation}
Then if $q>R^{3/2}$, $\frac{1}{\delta} \mathbf{P}_{\lambda,\mu}(\bw_{\delta}, \bw_{\delta})$ is bounded.

By summarising our earlier deduction, we clearly have shown that $\mathbf{I}_{\delta}(\bv_{\delta},\bw_{\delta})$ is bounded. That is, the configuration $(\mathbf{C}_{\widetilde\lambda,\widetilde\mu},\bff)$ is non-resonant. 

The proof is complete.

\end{proof}

\section{Concluding remarks}

In this paper, we consider the plasmon resonance for the elastostatic system in $\mathbb{R}^3$. First, we show that the plasmon device in the literature which induce resonance in $\mathbb{R}^2$ does not induce resonance in $\mathbb{R}^3$. Then, we derive two novel plasmon devices, one with a core and the other one without a core, such that plasmon resonances can occur. This is mainly based on the highly nontrivial derivation of the proper plasmon parameters in Theorem~\ref{thm:perfectwaves}. For the novel constructions, we show both resonance and non-resonance results for a very broad class of distributional sources based on variational arguments. Furthermore, we establish the dependence of the plasmon resonances on the location of the force term. That is, there exists a critical radius such that when the source is located within the critical radius, then resonance occurs, whereas if the source is located outside the critical radius, then resonance does not occur. By comparing our main results in Theorems~\ref{thm:nonresoance}--\ref{thm:main3} with those resonance and non-resonance results in the two-dimensional case in \cite{AKKY,LiLiu2d}, the plasmon resonances in the three-dimensional case reveal some distinct and novel behaviours. The variational approach was initiated in \cite{Klsap} for the plasmon resonance in the electrostatics modelled by the Laplace equation. As can be seen that the variational approach is well situated to treat the resonance and non-resonance of the plasmon devices; whereas the spectral approach initiated in \cite{Ack13} for the elastostatics can be used to show both resonance and non-resonance as well as their localised and cloaking effects. However, for the latter approach, it requires exact spectral information of the associated Neumann-Poincar\'e operator, which is not always available. In this paper, we adopt the spectral approach to treat the much more complicated and challenging plasmon resonances in the elastostatics modelled by the Lam\'e system. Similar to the electrostatics in \cite{Klsap,LLL}, we could only establish the resonance and non-resonance results, as well as their dependence on the source location, and we didn't show the localised and cloaking effects. Nevertheless, as pointed out in Remark~\ref{rem:NP}, using our results in Theorem~\ref{thm:perfectwaves}, one can derive the spectral properties of the Neumann-Poincar\'e operator in \eqref{eq:NPn1}. Then, by following the spectral arguments, one may also be able to show the localised and cloaking effects for the plasmon resonance in the three-dimensional elastostatics. We shall report those findings in our future work.

\section*{Acknowledgement}

The work was supported by the FRG grants from Hong Kong Baptist University, Hong Kong RGC General Research Funds, 12302415 and 405513, and the NSF grant of China, No. 11371115.


\begin{thebibliography}{99}


\bibitem{Acm13}
H.~Ammari, G.~Ciraolo, H.~Kang, H.~Lee, and G.W. Milton, \emph{Anomalous
  localized resonance using a folded geometry in three dimensions,}, Proc. R.
  Soc. A, \textbf{469} (2013), 20130048.

\bibitem{Ack13}
H.~Ammari, G.~Ciraolo, H.~Kang, H.~Lee, and G.W. Milton, \emph{Spectral theory of a Neumann-Poincar\'{e}-type operator and
  analysis of cloaking due to anomalous localized resonance}, Arch. Ration.
  Mech. Anal., \textbf{208} (2013), 667--692.

\bibitem{Ack14}
H.~Ammari, G.~Ciraolo, H.~Kang, H.~Lee, and G.W. Milton, \emph{Spectral theory of a Neumann-Poincar\'{e}-type operator and analysis of cloaking due to anomalous localized resonance {I}{I}},
  Contemporary Math., \textbf{615} (2014), 1--14.

\bibitem{ADM} {H.~Ammari, Y. Deng and P. Millien}, \emph{Surface plasmon resonance of nanoparticles and applications in imaging},
Arch. Ration. Mech. Anal., \textbf{220} (2016), 109--153.

\bibitem{AMRZ} {H.~Ammari, P. Millien, M. Ruiz and H. Zhang}, \emph{Mathematical analysis of plasmonic nanoparticles: the scalar case}, preprint, arXiv:1506.00866

\bibitem{ARYZ} {H.~Ammari, M. Ruiz, S. Yu and H. Zhang}, \emph{Mathematical analysis of plasmonic resonances for nanoparticles: the full Maxwell equations}, preprint, arXiv:1511.06817

\bibitem{AK} K. Ando and H. Kang, {\it Analysis of plasmon resonance on smooth domains using spectral properties of the Neumann-Poincar\'e operator}, J. Math. Anal. Appl., \textbf{435} (2016), 162--178.

\bibitem{AKKY} {K. Ando, Y. Ji, H. Kang, K. Kim and S. Yu}, \emph{Spectral properties of the Neumann-Poincar\'e operator and cloaking by anomalous localized resonance for the elastostatic system}, preprint, arXiv:1510.00989

\bibitem{AKL} {K. Ando, H. Kang and H. Liu}, \emph{Plasmon resonance with finite frequencies: a validation of the quasi-static approximation for diametrically small inclusions}, SIAM J. Appl. Math., \textbf{76} (2016), 731--749.

\bibitem{Bos10}
G.~Bouchitt\'{e} and B.~Schweizer, \emph{Cloaking of small objects by anomalous localized resonance}, Quart. J. Mech. Appl. Math., \textbf{63} (2010), 438--463.

\bibitem{Brl07} O.P. Bruno and S.~Lintner, \emph{Superlens-cloaking of small dielectric bodies in the quasistatic regime}, J. Appl. Phys., \textbf{102} (2007), 124502.

 \bibitem{CKKL} D. Chung, H. Kang, K. Kim and H. Lee, {\it Cloaking due to anomalous localized resonance in plasmonic structures of confocal ellipses}, preprint, arXiv: 1306.6679.

\bibitem{KLO} H. Kettunen, M. Lassas and P. Ola, {\it On absence and existence of the anomalous localized resonace without the quasi-static approximation}, preprint, arXiv: 1406.6224.

\bibitem{KM} {D. M. Kochmann and G. W. Milton}, \emph{Rigorous bounds on the effective moduli of
composites and inhomogeneous bodies with negative-stiffness phases}, J. Mech. Phys.
Solids, {\bf 71} (2014), 46--63.

\bibitem{Klsap} {R.V. Kohn, J.Lu, B.~Schweizer and M.I. Weinstein}, \emph{A variational
  perspective on cloaking by anomalous localized resonance}, Comm. Math. Phys., {\bf 328} (2014), 1--27.

  \bibitem{Kup} V. D. Kupradze, {\it Three-dimensional Problems of the Mathematical Theory of Elasticity and Thermoelasticity}, Amsterdam, North-Holland, 1979.

\bibitem{LLBW} {R.S. Lakes, T. Lee, A. Bersie, and Y. Wang}, \emph{Extreme damping in composite materials
with negative-stiffness inclusions}, Nature, {\bf 410} (2001), 565--567.

\bibitem{LiLiu2d} {H. Li and H. Liu}, \emph{On anomalous localized resonance for the elastostatic system}, arXiv:1601.07744

\bibitem{LLL} {H. Li, J. Li and H. Liu}, \emph{On quasi-static cloaking due to anomalous localized resonance in $\mathbb{R}^3$}, SIAM J. Appl. Math., {\bf 75}  (2015), no. 3, 1245--1260.

\bibitem{LASS} {A. E. H. Love}, \emph{A Treatise on the Mathematical Theory of Elasticity}, 4th Edition, Cambridge University Press, 2013.

\bibitem{GWM1} {R.C. McPhedran, N.-A.P. Nicorovici, L.C. Botten and G.W. Milton}, \emph{Cloaking by plasmonic
resonance among systems of particles: cooperation or combat?} C.R. Phys., {\bf 10} (2009), 391--399.

\bibitem{GWM2} {D. A. B. Miller}, \emph{On perfect cloaking}, Opt. Express, {\bf 14} (2006), 12457--12466.

\bibitem{GWM3} {G.W. Milton and N.-A.P. Nicorovici}, \emph{On the cloaking effects associated with anomalous
localized resonance}, Proc. R. Soc. A, {\bf 462} (2006), 3027--3059.

\bibitem{GWM4} {G.W. Milton, N.-A.P. Nicorovici, R.C. McPhedran, K. Cherednichenko and Z. Jacob},
\emph{Solutions in folded geometries, and associated cloaking due to anomalous resonance}, New. J. Phys., {\bf 10} (2008), 115021.


\bibitem{GWM6} {N.-A.P. Nicorovici, R.C. McPhedran, S. Enoch and G. Tayeb}, \emph{Finite wavelength cloaking
by plasmonic resonance}, New. J. Phys., {\bf 10} (2008), 115020.

\bibitem{GWM7} {N.-A.P. Nicorovici, R.C. McPhedran and G.W. Milton}, \emph{Optical and dielectric properties
of partially resonant composites}, Phys. Rev. B, {\bf 49} (1994), 8479--8482.

\bibitem{GWM8} {N.-A.P. Nicorovici, G.W. Milton, R.C. McPhedran and L.C. Botten}, \emph{Quasistatic cloaking
of two-dimensional polarizable discrete systems by anomalous resonance}, Optics
Express, {\bf 15} (2007), 6314--6323.

\bibitem{GWM9} G.W. Milton and N.-A.P. Nicorovici, \emph{On the cloaking effects associated
 with anomalous localized resonance}, Proc. R. Soc. A, \textbf{462} (2006),
  3027--3059.



\end{thebibliography}
\end{document}